\documentclass[reqno, A4paper]{amsart}
\usepackage{amsfonts, amsmath, amsthm, amssymb, latexsym, xfrac, mathrsfs, mathtools, float, xcolor}

\usepackage{tikz}
\usetikzlibrary{math}

\usepackage[
  headheight=14pt,
  top=1.45in,bottom=1.35in,left=1.14in,right=1.58in,
  twoside,
  asymmetric,
]{geometry}
\usepackage{hyperref, xcolor}
\hypersetup{
    colorlinks = false,
    linkbordercolor = {white},
    pdfauthor=author
}

\usepackage[utf8]{inputenc}

\usepackage{parskip}

\usepackage[font=scriptsize]{caption}
\usepackage{subcaption}

\newcommand\opint{\ensuremath{]0,\infty[}}
\newcommand\hopint{\ensuremath{[0,\infty)}}

\newcommand\wt{\widetilde}
\newcommand\N{\ensuremath{\mathbb{N}}}
\newcommand\R{\ensuremath{\mathbb{R}}}
\newcommand\Z{\ensuremath{\mathbb{Z}}}
\renewcommand{\H}{\mathcal{H}}

\newcommand\C{\ensuremath{\mathbb{C}}}

\newcommand\T{\ensuremath{\mathbb{T}}}
\renewcommand\a{\mathfrak{a}}

\DeclareMathOperator{\Hom}{Hom}

\newcommand{\ve}{\varepsilon}
\renewcommand{\Im}{\text{Im}}

\newcommand{\ol}{\overline}

\newcommand\weylcham{\ensuremath{\overline{\a_{+}}}}

\newcommand{\wtlatt}{\ensuremath{P}}
\newcommand{\domwt}{\ensuremath{P_{+}}}
\newcommand{\rtlatt}{\ensuremath{Q}}
\newcommand{\crtlatt}{\ensuremath{\rtlatt^{\vee}}}

\newcommand{\alc}{\ensuremath{A_{0}}}

\DeclareMathOperator{\spann}{span}

\makeatother

\usepackage[nameinlink,capitalize]{cleveref}%,german

\makeatletter
\@namedef{subjclassname@2020}{\textup{2020} Mathematics Subject Classification}
\makeatother

\theoremstyle{plain}
\newtheorem{theorem}{Theorem}[section]
\newtheorem{corollary}[theorem]{Corollary}
\newtheorem{lemma}[theorem]{Lemma}
\newtheorem{proposition}[theorem]{Proposition}
\theoremstyle{definition}
\newtheorem{definition}[theorem]{Definition}
\theoremstyle{remark}
\newtheorem{remark}[theorem]{Remark}
\newtheorem{remarks}[theorem]{Remarks}

\newtheorem{example}[theorem]{Example}
\numberwithin{equation}{section}

\let\phi\varphi

\title[Littlewood-Paley theory associated with root systems]{Littlewood-Paley theory for orthogonal expansions associated with root systems}
\author{Lukas Langen and Margit R\"osler}
\address{Institut f\"ur Mathematik, Universit\"at Paderborn, Warburger Str. 100, D-33098 Paderborn, Germany}
\email{llangen@math.upb.de, roesler@math.upb.de}


\thanks{The authors were supported by the  German Research Foundation (DFG), via the grant SFB-TRR  358/1 2023-491392403.}

\begin{document}

\date{\today}

\begin{abstract}
  We introduce the non-symmetric heat and Poisson semigroups associated with the Heckman-Opdam Laplacian in the compact setting. Based on the Poisson semigroup, we study several Littlewood-Paley $g$-functions in the spirit of Stein's work for compact Lie groups and prove their $L^{p}$-boundedness for $1<p\leq 2$ and for some of them also for $1<p<\infty$.
  As an application, we define associated Riesz transforms and imaginary powers and prove their $L^{p}$-continuity for $1<p<\infty$.
  Passing to the average with respect to the action of the associated reflection group, we obtain $L^{p}$-boundedness of $g$-functions for the symmetric Poisson semigroup for all $1<p<\infty$. In particular, our framework covers the Littlewood-Paley-Stein theory for Jacobi polynomial expansions and corresponding direct product settings as special cases.
\end{abstract}

\maketitle
\section{Introduction}

Since the seminal work of Muckenhoupt and Stein \cite{MS65} for classical orthogonal expansions and the monograph of Stein \cite{St70}, treating both compact Lie groups as well as a general framework related to diffusion semigroups,  there has been a continuous development of Littlewood-Paley theory for orthogonal expansions beyond the classical Fourier setting. Among the broad literature, we mention \cite{CS79, Th93, ST03, NSt06, NS08, NS12}. 

In this paper, we initiate the development of Littlewood-Paley theory for orthogonal expansions in the compact  Heckman-Opdam setting, also known as Dunkl theory in the compact trigonometric case; see \cite{HS94, Op95, Op00, HO21}. This theory simultaneously  generalizes  the radial analysis on Riemannian symmetric spaces of the compact type on one hand, and the theory of classical Fourier series as well as Jacobi polynomial expansions in one variable on the other. The role of invariant differential operators (such as the Laplace-Beltrami operator) on a Riemannian symmetric space is taken on by Dunkl-Cherednik operators, which are commuting differential-reflection operators associated with some root system on a Euclidean space. 
Besides trigonometric Dunkl theory in its compact and non-compact variants, there is a particularly  rich harmonic analysis available in rational Dunkl theory, which generalizes radial analysis on symmetric spaces of Euclidean type; see \cite{Du89, dJ93} and  \cite{RV08, A17} for an overview. In the rational Dunkl setting, many aspects of Littlewood-Paley theory, in particular $g$-functions and Riesz transforms,  have been extensively studied, see \cite{S05, TX07, ADH19, AS12, LZL17, DH22, H23, LZ23, DW26}.
In contrast, the trigonometric Dunkl setting has so far been much less explored. See however \cite{BS11}, which deals with the Hilbert transform in rank one, and the recent paper \cite{K25} on spectral and Fourier multipliers in the general Heckman-Opdam setting. 

In the present paper, we study $g$-functions and Riesz transforms associated with multivariate orthogonal expansions in terms of Heckman-Opdam polynomials as introduced in \cite{Op95}. See also \cite{HS94, HO21}, where symmetric Heckman-Opdam polynomials are treated, which are obtained as Weyl-group symmetrizations of the non-symmetric ones. 
Heckman-Opdam polynomials are associated with a root system $R$ on a finite-dimensional Euclidean space $\mathfrak a$ and a continuous family of parameters, given by a so-called multiplicity function $k$ on $R$ which is required to be invariant under the Weyl group  of $R$. The non-symmetric Heckman-Opdam polynomials are characterized as the joint polynomial eigenfunctions of an associated family of Dunkl-Cherednik operators. Indexed by the weight lattice of $R$, these polynomials form an orthogonal basis 
of the $L^2$-space on the compact torus $\mathbb T= \mathfrak a/2\pi Q^\vee$, where $Q^\vee $ is the coroot lattice of $R$, with respect to a certain trigonometric weight function depending on $R$ and the multiplicity $k.$ Due to a lack of suitable kernel bounds, there is  currently no established Calder\'on-Zygmund theory in the general compact Heckman-Opdam setting; this is in contrast to  the framework of Jacobi expansions, see \cite{NS12, NS13}. We therefore follow the classical approach of \cite{St70}. In our setting, the role of left-invariant vector fields in \cite{St70} is taken on by Dunkl-Cherednik operators. Indeed, the full algebra of Dunkl-Cherednik operators is generated by the first-order Dunkl-Cherednik operators, and in particular, the Heckman-Opdam Laplacian is naturally built up from first-order Dunkl-Cherednik operators similar to how the classical Laplacian is built up from partial derivatives. This simple fact turns out to be crucial for our studies. However, the analysis becomes more involved compared to \cite{St70}  due to the additional reflection terms in our operators.  
We mention that a decomposition of the relevant second-order operator determining a family of orthogonal polynomials into simpler first-order parts, though of somewhat different flavour, was also already at the heart of \cite{NSt06, NS08, FSS15}. 

Let us highlight some special cases covered by our general framework. Trivially, for the multiplicity $k=0$, Dunkl-Cherednik operators coincide with usual partial derivatives and  one recovers classical Littlewood-Paley theory on the torus $\T\cong \R^{n}/2\pi\Z^{n}$. In the rank 1 case, that is for the root system $BC_{1}\subseteq\a =\R$, the symmetric Heckman-Opdam polynomials are just classical Jacobi polynomials. In particular, by symmetrizing one obtains Littlewood-Paley theory for Jacobi expansions with arbitrary parameters $\alpha,\beta\geq -1/2$. Furthermore, the direct product situation $R=BC_{1}\times\ldots\times BC_{1}\subseteq\R^{n}$ naturally covers the multi-dimensional Jacobi expansion as studied in \cite{NS08}. We refer  to Example \ref{example:rank1} for further details. Further, spherical functions of Riemannian symmetric spaces $U/K$ of the compact type are given by symmetric Heckman-Opdam polynomials with certain half-integer multiplicities $k$ arising from root space decompositions, see \cite{HS94,HO21, RR15}. Thus, our framework in particular covers Littlewood-Paley theory for radial functions, that is $K$-invariant functions, on $U/K$.

The organization of this paper is as follows: In \cref{sec:prelim}, we introduce notations and recapitulate the basic concepts of compact Heckman-Opdam theory.
In \cref{sec:heat} we introduce the heat semigroup in the general compact Heckman-Opdam setting. This is a direct generalization of the symmetric (Weyl group invariant) heat semigroup that was studied in \cite{RR11}. By subordination, we then define the Poisson semigroup in \cref{sec:poisson} and study its basic properties. Both semigroups are symmetric diffusion semigroups in the sense of Stein, which is instrumental for the implementation of general Littlewood-Paley theory for $g$-functions involving only time-derivatives. 

\cref{sec:gfun} makes up the main part of this paper. We introduce several Littlewood-Paley $g$-functions in the spirit of Stein, part  of them involving Dunkl-Cherednik operators on the spatial side, and prove their  $L^{p}$-boundedness for all $1<p\leq 2$. The different $g$-functions are intimately connected,  and their interplay allows for dealing with reflection terms which naturally appear in the non-symmetric theory. Here we  employ some ideas of \cite{S05} and \cite{LZL17} in the rational Dunkl case.  Next, in \cref{sec:riesz}, we use our previously obtained results to define natural Riesz transforms and Riesz potentials associated with the Heckman-Opdam Laplacian and prove their $L^{p}$-boundedness for all $1<p<\infty$. In turn, this proves $L^{p}$-boundedness for some of the $g$-functions involving Dunkl-Cherednik operators on the spatial side considered in the previous section.
Finally, in \cref{sec:symmetric} we specialize to the $W$-invariant (also called \emph{symmetric}) setting. Here, we obtain $L^{p}$-boundedness of $g$-functions for all $1 < p <\infty$ as a direct consequence of our previous results using the non-symmetric Riesz transforms. This reveals how the non-symmetric theory simultaneously enriches and generalizes the symmetric setting. As special cases of the symmetric setting we obtain results for multi-dimensional Jacobi expansions and radial harmonic analysis on Riemannian symmetric spaces of the compact type.

\section{Cherednik Operators and Heckman-Opdam Polynomials}\label{sec:prelim}

In this section, we develop basic facts from Dunkl-Cherednik theory and the associated orthogonal polynomials. For a general background on this topic, the reader is referred to \cite{HS94, Op95, Op00, HO21}. We also include some additional fundamental properties needed later on.

Let $R$ be a crystallographic, but not necessarily reduced root system in a Euclidean space $(\a, \langle \cdot, \cdot\rangle)$ of finite dimension $n.$  We write $|x|=\sqrt{\langle x, x\rangle}$ for the associated norm and extend $\langle \cdot, \cdot\rangle $ in a bilinear way to the complexification $\a_{\C}$ of $\a$. We identify $\a$ with its dual $\a^{\ast}=\Hom(\a,\, \R)$ via the given inner product. For $\alpha\in R$ the reflection in the hyperplane $\alpha^{\perp}\subseteq\a$ is given by $s_{\alpha}(x)=x-\langle \alpha, x\rangle \alpha^{\vee}$, where $\alpha^{\vee}=\frac{2\alpha}{\langle \alpha, \alpha\rangle }$ denotes the coroot of $\alpha$. Let $W$ be the associated Weyl group, which is generated by the reflections $s_{\alpha},\,\alpha\in R$. The action of $W$ on $\a$ extends naturally to $\a_{\C}$ and thus to functions $f\colon\a_{\C}\to\C$ via $w.f(z)=f(w^{-1}z)$. We fix a positive subsystem $R_{+}$ of $R$ as well as a non-negative multiplicity function $k$ on $R$, that is a $W$-invariant function $k\colon R\to [0,\infty[,\ \alpha\mapsto k_{\alpha}$. For $\xi\in\a_{\C}$, the (Dunkl-) Cherednik operator $T_{\xi}=T_{\xi}(R_{+}, k)$ on $\a_{\C}$ is defined by
\begin{align*}
  T_{\xi}f = \partial_{\xi}f - \langle \rho, \xi\rangle f + \sum_{\alpha\in R_{+}} k_{\alpha} \langle \alpha, \xi\rangle \frac{f-s_{\alpha}f}{1-e^{-\alpha}}, 
\end{align*}
where for $\lambda\in\a_{\C}$ the exponential polynomial $e^{\lambda}$ is defined by 
$$e^{\lambda}(z)=e^{\langle \lambda, z\rangle }\ (z\in\a_{\C}),$$ $ \partial_\xi$ denotes the directional derivative  $\,\partial_\xi f(z) = \lim_{h\in \mathbb R, h\to 0} \frac{f(z+h\xi)-f(z)}{h},$
 and $$\rho=\rho(R_{+}, k):=\frac 1 2 \sum_{\alpha\in R_{+}}k_{\alpha}\alpha$$ is the generalized Weyl vector. It is contained in the topological closure of the positive chamber
 $$ \frak a_+=\{ x\in \frak a: \langle \alpha, x\rangle > 0 \ \forall\alpha \in R_+\}.$$
 We further introduce the coroot lattice 
$\crtlatt\coloneqq \spann_{\Z}\{\alpha^{\vee}:\alpha\in R\}$ and the weight lattice 
$$P:= \{ \lambda\in\a : \langle \lambda, \alpha^{\vee}\rangle \in\Z\ \forall\alpha\in R\}. $$ 
The weight lattice gives rise to the algebra of exponential polynomials $ \spann_{\C}\{e^{\lambda} : \lambda\in \wtlatt\}$, which naturally identifies with the group algebra $ \mathbb C[P].$  The operators $T_\xi, \, \xi \in \frak a$ leave this space of exponential polynomials invariant and commute, c.f. \cite{Op95}. 
 Note that  the $e^{\lambda},\, \lambda\in \wtlatt$ are $2\pi i \crtlatt$-periodic and constitute the group characters  of the compact torus $\,i\a/2\pi i \crtlatt$.
 In contrast to the references above,  we prefer to work on the real torus 
 $$ \mathbb T:= \frak a/2\pi Q^\vee,$$
whose dual group is $\{ e^{i\lambda}: \lambda \in P\}.$ 
 We therefore introduce for $\xi \in \frak a$   the modified Cherednik operators
 \begin{equation} \label{Cher_mod} D_\xi f := \partial_\xi f -i\langle \rho, \xi\rangle f       + \, \sum_{\alpha\in R_+} ik_\alpha \langle \alpha, \xi\rangle \,d_\alpha f \end{equation}
 on $\frak a_{\mathbb C}, $ where
 $$ d_\alpha f = \frac{f-s_\alpha f}{1-e^{-i\alpha}}.$$
 For a differentiable function $F$ on $i\frak a, $ the function   $f (x): = F(ix)$ defined on $\frak a$ then satisfies \begin{equation}\label{umrechnung}
(D_\xi f \,)(x) = (T_{i\xi} F) (ix).\end{equation}
   Consider the space of trigonometric polynomials 
   \begin{equation}\label{trigpol} 
\mathcal T := \text{span}_{\mathbb C}\{e^{i\lambda}, \, \lambda \in P\}
\end{equation}
as (smooth) functions on $\mathbb T,$ given by $e^{i\lambda}(x) = e^{i\langle \lambda, x\rangle}.$ The Cherednik operators $D_\xi, \, \xi \in \frak a$ leave $\mathcal T$ invariant and commute as well. Moreover, they behave
in a natural way on spaces of differentiable functions on $\mathbb T$. Indeed, we have

\begin{lemma}\label{diff_properties} Let $f \in C^m(\mathbb T)$ with $m\in \mathbb N$. Then for $\xi \in \frak a,$ $D_\xi f \in 
C^{m-1}(\mathbb T).$ Furthermore, the operators $D_{\xi}\colon C^{m}(\T)\to C^{m-1}(\T)$ are continuous with respect to the usual topologies.
 \end{lemma}

 \begin{proof} 
We may regard $f\in C^m(\mathbb T)$ as a function from $C^m(\frak a)$ which is $2\pi Q^\vee$-periodic. 
It suffices to analyze the operators $d_\alpha$. We need to see what happens at the points $x_{0}\in \a$ with $\langle \alpha, x_{0}\rangle = 2\pi  k, k\in\Z$. So consider $x\in \a$ close to $x_{0}$ with $\langle \alpha, x\rangle \notin 2\pi \Z$.  
 Then by the fundamental theorem of calculus, 
\begin{align*} -d_\alpha f (x) & = \frac{f(s_\alpha x) - f(x -2\pi k \alpha^\vee)}{1 - e^{-i\langle \alpha, x\rangle} } \\ & =  \frac{1}{1- e^{-i\langle \alpha, x\rangle} }      \int_0^1 \frac{d}{dt} 
f(x - 2 \pi k\alpha^\vee + t(2\pi k - \langle \alpha, x\rangle ) \alpha^\vee )dt \\ 
& = \frac{2\pi k - \langle \alpha, x\rangle}{1- e^{-i\langle \alpha, x\rangle} }
 \cdot \int_0^1 (\partial_{\alpha^\vee}\!f)\bigl(x - 2 \pi k\alpha^\vee + t(2\pi k - \langle \alpha, x\rangle )\alpha ^\vee \bigr)dt.
\end{align*}
As the first factor is smooth in a neighborhood of $x_0$, this shows that $\,d_\alpha f\in C^{m-1}(\mathbb T)$ for $f\in C^m(\mathbb T),$ and  the continuity statement follows as well.  
\end{proof}

Following \cite{Op95}, we equip  $\mathbb T$ with the $W$-invariant weight function
\begin{align*}
  \delta_{k}(x)\coloneqq \prod_{\alpha\in R_{+}}\left\vert 2\sin\frac{\langle \alpha, x\rangle}{2}\right\vert^{2k_{\alpha}} = \prod_{\alpha\in R} \big\vert 1-e^{-i\langle\alpha,x\rangle}\big\vert^{k_\alpha}
\end{align*}
and write
\begin{align*}
  \langle f, g\rangle_{\delta_{k}} \coloneqq \int_{\mathbb T} f(x)\ol{g(x)}\,\delta_{k}(x) dx
\end{align*}
for the scalar product on $L^{2}(\mathbb T,\delta_{k})$. Here $dx$ denotes the normalized Haar measure of $\mathbb T$.

\begin{remark}\label{rem:funDomain}
  Let $\{\alpha_{1}^{\vee},\ldots, \alpha_{n}^{\vee}\}\subseteq \crtlatt$ be a $\Z$-basis for $\crtlatt$. For example, if $R$ is reduced one could take $\{\alpha_{1},\ldots,\alpha_{n}\}\subseteq R$ a simple system in $R$ and $\alpha_{j}^{\vee}$ the corresponding coroots (where $n=\text{dim} \,\frak a$). Then a fundamental domain for $\mathbb T$ is given by the set $\,\{\sum_{j=1}^{n}t_{j}\alpha^{\vee}_{j} : 0\leq t_{j}< 2\pi\text{\ for all\ } j\}$, and
  \begin{align*}
    \int_{\T}f(x) dx = \frac{1}{(2\pi)^{n}}\int_{[0, 2\pi[^{n}} f\Big(\sum_{j=1}^{n}t_{j}\alpha_{j}^{\vee}\Big)\ dt_{1}\ldots dt_{n}. 
  \end{align*}
\end{remark}

\begin{lemma}\label{antisymm}
  The (modified) Cherednik operators $D_{\xi}, \xi\in \a$ are anti-symmetric on $C^{1}(\mathbb T)$ with respect to $\langle \cdot, \cdot\rangle_{\delta_k}$, i.e.
  \begin{align*}
    \langle D_{\xi}f, g\rangle_{\delta_k} = -\langle f, D_{\xi}g\rangle_{\delta_k}\quad\text{for all}\ f, g\in C^{1}(\mathbb T). 
  \end{align*}
\end{lemma}
\begin{proof}
  This follows from \cite[Prop. 2.3]{Op95}, but as the reference \cite[Prop. 3.8]{Ch91} cited there (implicitly) requires $k_\alpha \in \mathbb N_0 ,$ we include an explicit proof for the reader's convenience. We first assume that  $k\geq 1,$ i.e. $k_\alpha \geq 1 $ for all $\alpha$, in which case $\delta_k \in C^1(\mathbb T).$ We consider a fundamental domain according to \cref{rem:funDomain}. Integration by parts w.r.t. the coordinate $t_{j}$ then gives
  \begin{align*}
    \int_{\mathbb T} \partial_{\alpha_{j}^{\vee}}f(x)\ol{g(x)}\,\delta_{k}(x)dx = -\int_{\mathbb T} f(x)\partial_{\alpha_{j}^{\vee}}\bigl(\ol{g(x)}\delta_{k}(x)\bigr) dx.
  \end{align*}
  By linearity this extends to arbitrary $\xi\in \a$ instead of $\alpha_{j}^{\vee}$. As
  \begin{align*}
    \partial_{\xi}\delta_{k}(x) = \delta_{k}(x)\cdot \sum_{\alpha\in R_{+}}k_{\alpha}\langle \alpha, \xi\rangle \cot\frac{\langle \alpha, x\rangle }{2}, 
  \end{align*}
  we obtain
  \begin{align*}
    \langle \partial_{\xi}f, g\rangle_{\delta_k} = -\langle f, \partial_{\xi}g\rangle_{\delta_k} - \sum_{\alpha\in R_{+}} k_{\alpha}\langle \alpha, \xi\rangle \bigg\langle f,\, g\cot\frac{\langle \alpha, \cdot\,\rangle}{2}\bigg\rangle_{\delta_k}. 
  \end{align*}
  Further,
  \begin{align*}
    \langle d_{\alpha}f, g\rangle_{\delta_{k}} = \Big\langle f,\, (1-s_{\alpha})\Bigl(\frac{g}{1-e^{i\alpha}}\Bigr)\Big\rangle_{\delta_k} = \Big\langle f,\, d_{\alpha}g +i g\cot\frac{\langle \alpha,\cdot\,\rangle }{2}\Big\rangle_{\delta_k}.  
  \end{align*}
  Putting things together, we obtain the assertion for $k\geq 1.$ The general case $k\geq 0$ follows by analytic extension.
\end{proof}

On $C^{2}(\mathbb T)$, we consider the (modified) Heckman-Opdam Laplacian
\begin{align}\label{Def_L_k}
  L_{k}\coloneqq \sum_{j=1}^{n} D_{\xi_{j}}^{2} +\,|\rho|^{2} 
\end{align}
with some orthonormal basis $(\xi_{j})_{1\leq j\leq n}$ of $\a$. It is independent of the choice of this basis; indeed, a similar calculation as in \cite{Sch08} shows that
\begin{align*}
  L_{k} = \Delta + \sum_{\alpha\in R_{+}} k_{\alpha}\cot\frac{\langle \alpha, \cdot\rangle}{2}\partial_{\alpha} - \sum_{\alpha\in R_+} k_\alpha \frac{|\alpha|^2} {4\sin^2\tfrac{\langle\alpha, \,\cdot\rangle}{2}}\bigl(1-s_\alpha)
\end{align*}
where $\Delta$ denotes the usual Laplacian on the torus $\mathbb T$.
We split $L_k$ into its differential part and its reflection part, 
\begin{align*}
L_k^{\text{diff}} & = 	\Delta + \sum_{\alpha\in R_{+}} k_{\alpha}\cot\frac{\langle \alpha, \cdot\rangle}{2}\,\partial_{\alpha}\,,\\
L_k^{\text{ref}} &= \sum_{\alpha\in R_+} k_\alpha \frac{|\alpha|^2} {4\sin^2\tfrac{\langle\alpha, \,\cdot\rangle}{2}}\bigl(1-s_\alpha),
\end{align*}
so that $\, L_k = L_k^{\text{diff}} - L_k^{\text{ref}}\,.$ 
The operator $L_k^{\text{diff}}$ generalizes the radial part of the Laplace-Beltrami operator on a Riemannian symmetric space of compact type, which is obtained for particular half-integer values of  $k,$ see \cite[Chap. II, Prop. 3.11]{H00}.

The analysis on the torus $\mathbb T$ with respect to the weight $\delta_k$ is governed by the 
so-called non-symmetric Jacobi polynomials (also called non-symmetric Heckman-Opdam polynomials) associated with $R_+$ and $k.$ For their definition, we recall the usual dominance order on $P,$ which is defined by $\mu\leq \lambda :\Longleftrightarrow \lambda -\mu\in \text{span}_{\mathbb N_0}(R_+)$. 
Following \cite{HO21, Op00}, we define a modified dominance order on $P$ as follows:
\begin{align*}
\,\mu\triangleleft\lambda\,\text{ if either }\mu_+<\lambda_+\text{ or if }\,\mu_+=\lambda_+\,\text{ and }\lambda<\mu.
\end{align*}
 Moreover, $\,\mu\trianglelefteq  \lambda$ if $\mu\triangleleft \lambda $ or $\mu=\lambda.$
Here, $\lambda_{+}$ denotes the unique element from the $W$-orbit $W.\lambda$ of $\lambda$ in $\a$ which is contained in the closed chamber $\weylcham$\,.
The non-symmetric Heckman-Opdam polynomials 
 $\{ E_\lambda\,: \,\lambda\in P\} \subset \mathcal T$ associated with $R_+$ and $k$ are then
uniquely  characterized by the two conditions (c.f. \cite[Sect. 8.3]{HO21}) \begin{enumerate}
  \item[\rm{(i)}] $E_{\lambda} = E_\lambda(k;.\,) = \sum_{\mu \trianglelefteq\lambda} c_{\lambda\mu}e^{i \mu}\,\,$ with $c_{\lambda\mu} = c_{\lambda\mu}(k)\in\mathbb R,\ c_{\lambda\lambda}=1$;
\item[\rm{(ii)}] $\langle E_{\lambda}, e^{i\mu}\rangle_{\delta_k}=0\, $ for all $\mu\in P$ with $\mu \triangleleft \lambda$.
\end{enumerate}
The uniqueness is immediate from the fact that  $\text{span}_{\mathbb R}\{e^{i\mu}: \mu \trianglelefteq\lambda\}$ is a finite-dimensional Hilbert space:  Indeed, the set $\{\mu\in P : \mu \trianglelefteq\lambda\}$ is contained in the convex hull of the $W$-orbit $W.\lambda$ of $\lambda$ in $\frak a,$ see e.g. \cite{RV04}. We also mention that \cite{Op95} uses a slightly different partial order on $P,$ which however leads to the same Jacobi polynomials $E_\lambda$. 

Note that $E_\lambda(-x) = \overline{E_\lambda(x)}$ for $x\in \mathbb T$ and that $E_\lambda(x) = e^{i\langle \lambda, x\rangle}$ if $k=0.$

\begin{remarks}\label{rem:sahi} 1. 
Due to a deep result of Sahi \cite{S00a}, the coefficients $c_{\lambda\mu}(k)$ are rational functions of $k$ with non-negative coefficients. In fact, \cite{S00a} shows this for reduced crystallographic root systems. The same proof works for the only non-reduced irreducible crystallographic root systems, namely those of type  $BC_{n}$, which follows from the recursions established in \cite{S00b}. Hence this property holds  for all irreducible crystallographic root systems. Finally, we may decompose an arbitrary crystallographic root system into a disjoint union of irreducible ones.  The associated Heckman-Opdam polynomials split into products of Heckman-Opdam polynomials with respect to the irreducible subsystems (\cite{B25}). Thus the $c_{\lambda\mu}(k)$ are rational functions of $k$ with non-negative coefficients and in particular, $c_{\lambda\mu}(k)\geq 0$. We also mention Appendix A of the dissertation \cite{B24diss} for an  indepth discussion of these matters.

2. The symmetric (i.e. $W$-invariant) Heckman-Opdam polynomials treated in \cite{HS94} are obtained from the non-symmetric ones by Weyl-group symmetrization. They are indexed by the set of dominant weights $P_+ = P \cap \overline{\mathfrak a_+}$ and given by
$$ P_\lambda(x) = \frac{|W\lambda|}{|W|}\sum_{w\in W} E_\lambda( w^{-1}x), \quad \lambda \in P_+.$$
\end{remarks}

The non-symmetric Heckman-Opdam polynomials are simultaneous eigenfunctions of the Cherednik operators $D_\xi$. More precisely,
by \cite{Op95}  and in view of \eqref{umrechnung}, 
\begin{equation}\label{Cherednik_EF} D_\xi E_\lambda =  i\langle\,\widetilde\lambda,\xi\rangle\, 
E_\lambda \quad (\xi\in \frak a, \, \lambda\in P)\end{equation}
with the shifted spectral variable 
$$ \widetilde\lambda =  \lambda + 
\frac{1}{2} \sum_{\alpha\in R_+} k_\alpha \epsilon(\langle\lambda,\alpha^\vee\rangle) \alpha,\,$$
 where $\epsilon: \mathbb R\to \{\pm 1\}$ is given by 
 $\epsilon(x) = 1$ for $x>0$, $\epsilon(x) = -1$ for $x\leq 0$. 
 Then according to Propos. 2.10. of \cite{Op95}, we have $\widetilde\lambda\in W.(\lambda_++\rho).$
Thus by the definition of the Heckman-Opdam Laplacian,
\begin{equation}\label{eigenvalues_L_k} L_k E_\lambda = -\theta_\lambda E_\lambda
	\end{equation}
with 
$$ \theta_\lambda = |\widetilde\lambda|^2 - |\rho|^2 = |\lambda_+ + \rho|^2 - |\rho|^2 = \langle \lambda_++2\rho, \lambda_+\rangle.$$

Actually, the  set $\{E_\lambda: \lambda \in P\}$ forms an orthogonal basis of $L^2(\mathbb T, \delta_k)$, c.f. \cite[Cor. 2.11]{Op95}. 

\begin{lemma}\label{ev} \begin{enumerate}
\item[\rm{(a)}] For any $x,y\in \overline{\frak a_+}\,,$ we have $\langle x,y\rangle \geq 0.$
\item[\rm{(b)}] $\theta_{\lambda}>0$ for $\lambda\in P\setminus\{0\}$ and $\theta_{0}=0$.
\end{enumerate}
\end{lemma}

\begin{proof}

(a) Let $\lambda_{1},\ldots, \lambda_{n}$ denote the fundamental dominant weights relative  to the positive simple roots $\alpha_{1},\ldots, \alpha_{n}\in R_+$, which are characterized by $\langle \lambda_{i}, \alpha_{j}^{\vee}\rangle =\delta_{ij}$. Then $ x = \sum_{i=1}^n c_i\lambda_i\,$ with $c_i =\langle x, \alpha_i^\vee\rangle \geq 0,$ and similar for $y.$ But as $\langle\alpha_i,\alpha_j\rangle \leq 0,$ the elements of the dual basis satisfy $\langle \lambda_i,\lambda_j\rangle \geq  0$. This implies part (a). 

For part (b), note that $\lambda_+\in \overline{\frak a_+}.$ Moreover,  it is  known that   $\rho\in\weylcham\,$, see e.g. Section 2 of \cite{RKV13}. Thus by part (a),
  \begin{align*}
    \langle \lambda_{+}+2\rho, \lambda_{+}\rangle = \langle\lambda_{+}, \lambda_{+}\rangle + \langle 2\rho, \lambda_{+}\rangle \geq 0.
  \end{align*}
  Furthermore, $\theta_{\lambda}=0$ implies that $ \lambda_{+} = 0,$ since $\langle 2\rho, \lambda_{+}\rangle \geq 0$.
\end{proof}

It will be convenient to renormalize the non-symmetric Heckman-Opdam polynomials according to
\begin{align*}
  R_{\lambda} \coloneqq \frac{1}{E_{\lambda}(0)}E_{\lambda} \quad\text{and define}\quad r_{\lambda}\coloneqq (\|R_{\lambda}\|_{L^{2}(\T, \delta_{k})}^{2})^{-1}. 
\end{align*}

Be aware that in \cite{RR11} the $R_{\lambda}$ denote the $W$-invariant renormalized Jacobi polynomials. Note also that 
$$ 	
R_{\lambda}=\sum_{\mu\trianglelefteq\lambda}a_{\lambda\mu}e^{i\mu} \quad \text{  with }\,a_{\lambda\mu}\geq 0\ \text{ and } \sum_{\mu\trianglelefteq\lambda}a_{\lambda\mu} = 1 $$  as a consequence of \cref{rem:sahi}.  As $|\mu|\leq |\lambda|$ for $\mu \trianglelefteq \lambda$, the $R_\lambda$ therefore satisfy the estimate 
\begin{align}\label{estimate_R} |R_\lambda(z)| \leq e^{|\lambda||\text{Im}(z)|} \quad \text{ for } z\in \frak a_{\mathbb C}.\end{align}
In particular, $\|R_{\lambda}\|_{\infty, \T}=1$. The above convex representation also shows that  all derivatives of $R_\lambda$ on $\mathbb T$
are of polynomial growth in $|\lambda|.$ More precisely,
\begin{equation}\label{R_ableitung} \|\partial^\alpha R_\lambda\|_{\infty, \mathbb T} \, \leq |\lambda|^{|\alpha|} \quad (\alpha \in \mathbb N_0^n) .\end{equation}

There is a natural generalization of the classical Fourier transform on $\mathbb T$ in the Heckman-Opdam setting, namely 
$$ \widehat f^{\,k}(\lambda) := \int_{\mathbb T} f(x)\overline{R_\lambda(x)}\delta_k(x)dx \quad \text{ for } f\in L^1(\mathbb T, \delta_k) \text{ and } \lambda \in P.$$
If $f\in C^1(\mathbb T),$ then by the anti-symmetry of the Cherednik operators $D_\xi$  together with the eigenvalue equation \eqref{Cherednik_EF}, we obtain
\begin{equation}\label{Cherednik_FT} (D_\xi f)^{\wedge k}(\lambda) = i\langle\widetilde\lambda,\xi\rangle\, \widehat f^{\,k}(\lambda) \quad (\xi \in \mathfrak a).\end{equation}

\section{The non-symmetric Heat Semigroup}\label{sec:heat}

We  introduce a non-symmetric analogue of the heat semigroup studied in \cite{RR11} by removing the restriction to $W$-invariant functions. Note that the scaling of the root system and the multiplicity function there differed by a factor $2$ from our present notation, c.f. Remark 2.3. of 
loc.~cit.

\begin{definition}
  The (non-symmetric) heat kernel $\Gamma_{k}$ on $\T$  is defined by
  \begin{align*}
    \Gamma_{k}(t, x, y)\coloneqq \sum_{\lambda\in\wtlatt} r_{\lambda}e^{-\theta_{\lambda}t} R_{\lambda}(x)R_{\lambda}(-y)
     \quad  (t\in\, ]0,\infty[, x,y\in \mathbb T).
  \end{align*}
\end{definition}

We proceed as in \cite{RR11} to prove convergence of this series.  Recall that each $\lambda\in P$ is of the form $\lambda=w\lambda_+$ with some $w\in W.$
\begin{lemma}
  There exists a constant $C>0$ such that for all $\lambda\in P_+$ and $w\in W,$ 
  \begin{align*}
    |r_{w\lambda}| \leq\, C\cdot\prod_{\alpha\in R_{+}, \lambda_{\alpha}\neq 0} \lambda_{\alpha}^{2k_{\alpha}}. 
  \end{align*}
  Here, $\lambda_{\alpha}\coloneqq \langle \lambda, \alpha^{\vee}\rangle\geq 0.$
\end{lemma}
\begin{proof}
  We need the $L^{2}$-norms as well as the evaluations at $0$ of the (non-symmetric) Jacobi polynomials, which  are explicitly known (see \cite{HO21}, Chapter 8). In fact, consider $\lambda\in\domwt$. Then for $w\in W$,
\begin{align*}
  E_{w\lambda}(0) = \frac{\wt c_{w_{0}}(\rho)}{\wt c_{w w_{\lambda}}(\lambda+\rho)} > 0
\end{align*}
where
\begin{align*}
  \wt c_{w}(\lambda) \coloneqq \prod_{\alpha\in R_{+}} \frac{\Gamma(\lambda_\alpha +\frac 1 2 k_{\alpha/2}+\delta_{w}(\alpha))}{\Gamma(\lambda_\alpha +\frac 1 2 k_{\alpha/2}+k_{\alpha}+\delta_{w}(\alpha))}\quad\text{with}\quad \delta_{w}(\alpha) \coloneqq\begin{cases}
    0 &\text{if}\ w\alpha\in R_{+},\\ 1 &\text{otherwise}.
  \end{cases}
\end{align*}
Here, $w_{0}$ denotes the longest element in $W$ and $w_{\lambda}$ the longest element in the isotropy group $W_{\lambda}\subseteq W$ of $\lambda$.
Furthermore,
\begin{align*}
  \|E_{w\lambda}\|_{2,\delta_{k}}^{2} = \frac{c^{*}_{w w_{\lambda}}(-\lambda-\rho)}{\wt c_{w w_{\lambda}}(\lambda+\rho)},
\end{align*}
where
\begin{align*}
  c^{*}_{w}(\lambda)\coloneqq\prod_{\alpha\in R_{+}} \frac{\Gamma(-\lambda_\alpha -\frac 1 2 k_{\alpha/2}-k_{\alpha}+\delta_{w}(\alpha))}{\Gamma(-\lambda_\alpha - \frac 1 2 k_{\alpha/2}+\delta_{w}(\alpha))}.
\end{align*}

  We thus calculate 
  \begin{align*}
    r_{w\lambda} &= \frac{E_{w\lambda}(0)^{2}}{\|E_{w\lambda}\|_{2,\delta_{k}}^{2}} =  \frac{\wt c_{w_{0}}(\rho)^{2}}{\wt c_{ww_{\lambda}}(\lambda+\rho)\, c^{*}_{ww_{\lambda}}(-\lambda-\rho)} = \wt c_{w_{0}}(\rho)^{2} \prod_{\alpha\in R_{+}} f_{\alpha}(\lambda_{\alpha})
  \end{align*}
  with
  \begin{align*}
    f_{\alpha}(\lambda_{\alpha}) = \frac{\Gamma(\lambda_{\alpha}+\rho_{\alpha}+\frac 1 2 k_{\alpha/2}+k_{\alpha}+\delta_{ww_{\lambda}}(\alpha))\, \Gamma(\lambda_{\alpha}+\rho_{\alpha}-\frac 1 2 k_{\alpha/2}+\delta_{ww_{\lambda}}(\alpha))}{\Gamma(\lambda_{\alpha}+\rho_{\alpha}+\frac 1 2 k_{\alpha/2}+\delta_{ww_{\lambda}}(\alpha))\, \Gamma(\lambda_{\alpha}+\rho_{\alpha}-\frac 1 2 k_{\alpha/2}-k_{\alpha}+\delta_{ww_{\lambda}}(\alpha))}.
  \end{align*}
  Using the asymptotics
  \begin{align*}
    \frac{\Gamma(z+a)}{\Gamma(z+b)}\sim z^{a-b}\quad\text{for}\,\, z\to\infty, 
  \end{align*}
  we thus get
  \begin{align*}
    f_{\alpha}(\lambda_{\alpha})\sim \lambda_{\alpha}^{2k_{\alpha}}.
  \end{align*}
  as  $\lambda_{\alpha}\to\infty$. We now continue as in \cite{RR11} to conclude the proof.
  \end{proof}

We thus obtain

\begin{proposition}
  The series defining $\Gamma_k$ converges absolutely on $\opint\times \a\times \a$ and uniformly on each subset $[\delta,\infty[\times \frak a\times \frak a$  with $\delta >0.$ The long-time behaviour of $\Gamma_{k}$ is given by
  \begin{align*}
    \lim_{t\to\infty} \Gamma_{k}(t, x, y) = r_{0} = \frac{1}{\,\int_{\T}\delta_{k}(\xi)d\xi\,}\,. 
  \end{align*}
\end{proposition}

Whenever convenient, we shall identify $\Gamma_{k}$ with a function on $\opint\times \a\times \a\,$ which is $2\pi \crtlatt$-periodic in $x$ and $y$. In view of estimate \eqref{estimate_R} for the polynomials $R_\lambda$ and with the same argument as in \cite[Prop. 3.5]{RR11} we obtain 

\begin{proposition}
  For fixed $t>0$ the heat kernel $\Gamma_{k}(t, \cdot, \cdot)$ extends to a holomorphic function on $\a_{\C}\times \a_{\C}$ which is $2\pi \crtlatt$-periodic in the real part of both arguments, namely
  \begin{align*}
    \Gamma_{k}(t, z, w) = \sum_{\lambda\in\wtlatt}r_{\lambda}e^{-\theta_{\lambda}t}R_{\lambda}(z)R_{\lambda}(-w). 
  \end{align*}
  In particular, $\Gamma_{k}\in C^{\infty}(\opint\times \T\times \T)$.
\end{proposition}

\begin{lemma}
  \begin{enumerate}
    \item For fixed $w\in\a_{\C}$ the function $g(x, t)\coloneqq\Gamma_{k}(t, x, w)$ is a solution to the heat equation $\,L_{k}g=\partial_{t}g\text{ on }\T\times\opint.$
\item  $\displaystyle\int_{\T} \Gamma_{k}(t, z, y)\, \delta_{k}(y)\, dy = 1\,$ for all $z\in \a_{\C}$.
\item  $\displaystyle\Gamma_{k}(t+s, z, w)=\int_{\T}\Gamma_{k}(t, z, y)\Gamma_{k}(s, y, w)\,\delta_{k}(y)\, dy\, $ for all $z, w\in\a_{\C}$.
\item $\displaystyle\int_{\T}\Gamma_{k}(t, z, y) R_{\lambda}(y)\,\delta_{k}(y)\, dy = e^{-\theta_{\lambda}t}R_{\lambda}(z)\,$ for all $z\in\a_{\C}$.
  \end{enumerate}
\end{lemma}

\begin{definition}
  For $f\in L^{1}(\T,\,\delta_{k})$ we define
  \begin{align*}
    H(t)f(x) \coloneqq\begin{cases}
      \int_{\T}\Gamma_{k}(t, x, y)f(y)\,\delta_{k}(y)\, dy &\text{ for}\ t>0,\\
      f(x)&\text{ for}\ t=0.
    \end{cases}
  \end{align*}
\end{definition}

We obtain the following result extending \cite{RR11}:
\begin{theorem}
  The family $(H(t))_{t\geq 0}$ is a Feller-Markov semigroup on $(C(\T),\,\|\cdot\|_{\infty})$, i.e. it is a contractive and positivity-preserving strongly continuous operator semigroup.   Its generator is given by the closure $\ol L_{k}$ of the operator $L_{k}$ with domain $\mathcal D(L_k)=\mathcal T.$ 

We call $(H(t))_{t\geq 0}$ the \emph{(non-symmetric) Heckman-Opdam heat semigroup on $\T$}.
\end{theorem}
\begin{proof}
  The proof is essentially the same as in \cite{RR11}, Thm. 3.9. There is, however, a slight inaccuracy in the proof of the positive maximum principle, which we fix here. According to \cite[Ch. 4, Thm. 2.2]{EK86}, we have to verify the following conditions:
  \begin{enumerate}\itemsep=-2pt
  \item[\rm{(i)}] 	If $f\in \mathcal D(L_k)$, then also $\overline f\in \mathcal D(L_k),$ and $L_k\overline f = \overline{L_k f}.$
\item[\rm{(ii)}] There exists some $t>0$ such that the range of $\,t-L_k$ is dense in $C(\mathbb T).$
 \item[\rm{(iii)}] (Positive maximum principle): Let $f\in\mathcal T$ be real-valued with a non-negative maximum attained at $x_{0}\in\T$. Then $L_{k}f(x_{0})\leq 0$.
 \end{enumerate}
  Condition (i) is obvious and (ii) is immediate from $(t-L_k)R_\lambda = (t+\theta_\lambda) R_\lambda.$
  Condition (iii) requires  closer inspection. We consider $f$ as a $2\pi \crtlatt$-periodic function on $\a$.  If $x_{0}\in \frak a$ is regular as an element of $\mathbb T$, i.e. $\langle \alpha, x_{0}\rangle \notin 2\pi\Z$ for all $\alpha\in R$, then the statement is obvious, since
  \begin{align*}
    L_{k}f(x_{0}) \leq L_{k}^{\operatorname{diff}} f(x_{0}) = \Delta f(x_{0})\leq 0. 
  \end{align*}
  Suppose now  $\langle \alpha, x_{0}\rangle = 2\pi k$ for some $k\in\Z$ and $\alpha\in R_{+}$. Let $x\in\a$ be close to $x_{0}$ and 
  regular. We make a  Taylor expansion at $x-2\pi k\alpha^{\vee}, $ keeping in mind the periodocity of $f$ and its derivatives. This yields 
  \begin{align*}
    f(s_{\alpha}x)-f(x) & =\,f(s_{\alpha}x)- f(x-2\pi k\alpha^{\vee})\\
                       &=\,  \big( 2\pi k-\langle \alpha, x\rangle \big)\partial_{\alpha^{\vee}}f(x) + \frac{1}{2} \big(2\pi k - \langle \alpha, x\rangle\big)^{2} (\alpha^\vee)^{T} D^{2}f(z_x)\alpha^\vee
                       \end{align*}
 with some $z_x$ on the line segment $[x-2\pi k\alpha^{\vee}, s_{\alpha}x]$.
  We thus obtain
  \begin{align*}
     g_\alpha(x) &:=  \cot\frac{\langle \alpha, x\rangle}{2}\partial_{\alpha}f(x) - \frac{|\alpha|^{2}}{4\sin^{2}\frac{\langle \alpha, x\rangle }{2}}\bigl(f(x)-f(s_{\alpha}x)\bigr) \\
     & = \Bigl( \cot\frac{\langle\alpha, x\rangle}{2} + \frac{2\pi k -\langle \alpha, x\rangle}{2\sin^2 
     \frac{\langle\alpha, x\rangle}{2}} \Bigr)\cdot\partial_\alpha f(x) 
      +\frac{1}{2|\alpha|^{2}}\frac{\bigl(2\pi k - \langle \alpha, x\rangle\bigr)^2}{\sin^{2}\frac{\langle \alpha, x\rangle }{2}}\alpha^{T}D^{2}\!f(z_x)\alpha \,\\
      &  =: \varphi_1(x)\partial_\alpha f(x) + \varphi_2(x)\,\alpha^{T}D^{2}\!f(z_x)\alpha     \end{align*}
    As $x\to x_0$, we have  $\langle\alpha, x\rangle \to 2\pi k\,, \varphi_1(x) \to 0$ and $\,\displaystyle \varphi_2(x) \to 2/|\alpha|^2.$
    Moreover, $s_\alpha x\to x_{0}-2\pi k\alpha^{\vee},$   and therefore the line segment  $[x-2\pi k\alpha^{\vee}, s_{\alpha}x]$ shrinks to the singleton $\{x_{0}-2\pi k\alpha^{\vee}\}.$  
    Again by the periodicity of $f$, we thus get
    $$ \lim_{x\to x_0} g_\alpha(x) =  \frac{2}{|\alpha|^2} \alpha^T D^2\!f(x_0)\alpha\, \leq 0.$$ 
    This shows that $L_kf(x_0) \leq 0.$ 
   \end{proof}

   With the same proof as for the $W$-invariant case in \cite{RR11}, we conclude
\begin{corollary}
  The heat kernel $\Gamma_{k}$ is strictly positive, i.e.
  \begin{align*}
    \Gamma_{k}(t, x,y) > 0 \qquad\text{for all}\ (t, x,y)\in\ \opint\,\times\,\T\times\T. 
  \end{align*}
\end{corollary}

Note that $\Gamma_{k}(t, x, y)=\Gamma_{k}(t, y, x)$ as $\Gamma_{k}$ is real-valued on the above domain and $\overline{R_{\lambda}(x)}=R_{\lambda}(-x)$ for all $x\in\T$. 
As $C(\mathbb T)$ is dense in each of the spaces $L^p(\mathbb T, \delta_k), \, 1\leq p <\infty$, standard arguments as in Proposition 3.11 of \cite{RR11} also imply the following

\begin{corollary}
  The family $(H(t))_{t\geq 0}$ defines a symmetric diffusion semigroup on $\mathbb T$ in the sense of \cite{St70}, i.e. it is a contractive operator semigroup on $L^p(\mathbb T)$ for $1\leq p \leq\infty$ which is Markovian, i.e. $H(t)1 = 1,$ and each $H(t)$ is self-adjoint on $L^2(\mathbb T).$  Moreover, the heat semigroup is strongly continuous for $1\leq p<\infty.$ 
 
\end{corollary}

\section{The Poisson Semigroup}\label{sec:poisson}

As usual, we define the Poisson semigroup $(P(t))_{t\geq 0}$ by subordination from $(H(t))_{t\geq 0}$ with respect to the convolution semigroup $(\sigma_{t})_{t\geq 0}$ on $\hopint$ with the probability measures
\begin{align*}
  d\sigma_{t}(s) = \frac{1}{\sqrt{4\pi}}ts^{-3/2}e^{-t^{2}/4s},\quad t\geq 0. 
\end{align*}

This means that for $f\in L^{p}(\T,\,\delta_{k})$ with $1\leq p \leq  \infty,$ $P(0)f:=f\,$ and for $t> 0$, $P(t)f\in L^{p}(\T,\,\delta_{k})$ is defined as the Bochner integral 
\begin{align*}
  P(t)f \coloneqq \int_{0}^{\infty} H(s)f\ d\sigma_{t}(s) = \frac{1}{\sqrt{\pi}}\int_{0}^{\infty}\frac{e^{-\tau}}{\sqrt{\tau}}H(t^{2}/4\tau)f\ d\tau. 
\end{align*}

By construction, $(P(t))_{t\geq 0}$ constitutes a symmetric diffusion semigroup.  It is strongly continuous on $C(\mathbb T)$ and on $L^p(\mathbb T, \delta_k)$ for $1\leq p<\infty$.  For $t>0$, the Poisson integral of $f\in L^{p}(\T,\,\delta_{k})$ is given explicitly by
\begin{align*}
  P(t)f(x) = \int_{\T}P_{k}(t, x, y) f(y)\,\delta_{k}(y) dy
\end{align*}
with the Poisson kernel
$$ P_{k}(t, x, y)= \int_{0}^{\infty}\Gamma_{k}(s, x, y)\ d\sigma_{t}(s), \quad t\in ]0,\infty[, x,y\in \mathbb T.$$ 
  The well-known identity 
 \begin{align*}
  \frac{1}{\sqrt{\pi}}\int_{0}^{\infty} \frac{e^{-\tau}}{\sqrt{\tau}}e^{-c^{2}/4\tau}\ d\tau = e^{-c}\quad\text{for\ } c>0. 
\end{align*}
leads to the expression
\begin{align}\label{poissonseries} 
 P_k(t,x,y) =  
  \sum_{\lambda\in \wtlatt}r_{\lambda}e^{-t\sqrt{\theta_{\lambda}}}R_{\lambda}(x)R_{\lambda}(-y).
  \end{align}
Note that $P_k>0.$ Further, as the constants $r_\lambda$ and the derivatives of the  $R_\lambda$ are of polynomial growth in $|\lambda|$ according to estimate \eqref{R_ableitung},  it is immediate that $P_k \in C^\infty(]0,\infty[ \times \mathbb T \times \mathbb T).$ Indeed,  the series and all its partial derivatives converge normally on  $[\delta, \infty[\times \mathbb T\times \mathbb T$ for each $\delta>0.$ 
Hence for $f\in L^p(\mathbb T, \delta_k), \, 1 \leq p \leq \infty,$ the  function $u(x, t):= P(t)f(x)$ belongs to $C^{\infty}(\T\times \opint)$ and solves the (generalized) Laplace equation $$(L_{k}+\partial_{t}^{2})u = 0 \quad\text{ on } \,\T\times \,\opint.$$ 
For $f\in L^1(\mathbb T, \delta_k)$ a short calculation shows that
$$ (P(t)f)^{\,\wedge k}(\lambda) = \, e^{-t\sqrt{\theta_\lambda}}\,\widehat f^{\,k}(\lambda) .$$
In particular,  for $f\in L^{2}(\T,\,\delta_{k})$ with orthogonal expansion
$\, f = \sum_{\lambda\in \wtlatt}a_{\lambda}R_{\lambda} \,$ 
it is immediate that
$\, P(t)f = \sum_{\lambda\in\wtlatt} e^{-t\sqrt{\theta_{\lambda}}}a_{\lambda}R_{\lambda}\,$ 
in $L^{2}(\T, \delta_{k})$.

We say that  a function $u\in C^{\infty}(\T\times\opint)$ belongs to $C^{\infty}(\T\times [0,\infty])$, if all its partial derivatives extend  continuously to $\T\times [0, \infty[$ and have a limit for $t\to \infty$. 
When starting with smooth initial data $f$, we obtain the following decay properties for $u_f(\cdot, t) =P(t)f$.

\begin{lemma}\label{lem:smoothAtTimeBoundary}
  For  $f\in C^{\infty}(\T)$, the function $u_f(x, t)=P(t)f(x)$ belongs to $C^{\infty}(\T\times [0,\infty]).$ Moreover, for each polynomial $p\in\C[t]$, $\alpha\in \mathbb N_0^n$ and strictly positive powers $m\in \mathbb N$  we have 
  $$\lim_{t\to\infty} p(t)\partial_{x}^{\alpha}\partial_{t}^{m}u_f(\cdot, t) =0,$$  
  the convergence being uniform on $\mathbb T.$ 
\end{lemma}
\begin{proof}
  We need to show smoothness of $u_f$ at the endpoints $t=0$ and $t=\infty.$ Let $\alpha\in \mathbb N_0^n, m\in \mathbb N_0.$ 
  By the dominated convergence theorem and the normal convergence of the series \eqref{poissonseries} and its partial derivatives on  $[\delta, \infty[\times \mathbb T\times \mathbb T$ with $\delta>0$ it follows that
  $$ \partial_{x}^{\alpha}\partial_{t}^{m} u_f(x, t) = \int_{\T} \Bigl(\sum_{\lambda\in\wtlatt} r_{\lambda}  \partial_{t}^{m}e^{-t\sqrt{\theta_{\lambda}}} \,\partial_{x}^{\alpha} R_{\lambda}(x) R_{\lambda}(-y)\Bigr) f(y)\delta_{k}(y) dy$$ 
  on $\mathbb T \times ]0,\infty[.$  Recall that the $\theta_\lambda$ are also of polynomial growth. 
For $\alpha\in\N_{0}^{n}, \,m\in \N_{0}$ we thus get
  \begin{align*}
    \lim_{t\to\infty} \partial_{x}^{\alpha}\partial_t^{m} u_f(x, t) &= \int_{\T} \sum_{\lambda\in\wtlatt} r_{\lambda}  (-\sqrt{\theta_\lambda})^m \lim_{t\to\infty} e^{-t\sqrt{\theta_{\lambda}}}\, \partial_{x}^{\alpha} R_{\lambda}(x) R_{\lambda}(-y) f(y)\delta_{k}(y) dy\\
    &=\begin{cases}
      \displaystyle r_{0}\int_{\T}f\delta_{k} dy & \text{if}\ \alpha=0\text{ and}\ m = 0,\\
      \, 0 &\text{otherwise}.
    \end{cases}
  \end{align*}
  This proves  $u\in C^{\infty}(\T\times\, ]0, \infty]).$ 
  
 The same calculation shows that even $\,\lim_{t\to\infty} p(t) \partial_{x}^{\alpha}\partial_{t}^{m}u_f(x, t) \,= 0\,$ for any polynomial $p$ as soon as $m> 0$, and the uniform convergence on $\mathbb  T$ is also obvious.  

  It remains to prove that all the limits 
  $\,\lim_{t\to 0} \partial_{x}^{\alpha}\partial_{t}^{m}u_f(x, t)\,$  with $ \alpha\in \mathbb N_0^n, m\in \mathbb N_0$ 
   exist pointwise on $\mathbb T.$ 
For initial data $f\in \mathcal T$ this is easy. Indeed, for $f=R_{\mu}$ we calculate 
  \begin{align*}
    u_f(x, t) &= \int_{\T}\sum_{\lambda \in \wtlatt} r_{\lambda}e^{-t\sqrt{\theta_{\lambda}}} R_{\lambda}(x) R_{\lambda}(-y)R_{\mu}(y)\,\delta_{k}(y)dy\\
            &= \sum_{\lambda\in\wtlatt} r_{\lambda}R_{\lambda}(x) e^{-t\sqrt{\theta_{\lambda}}} \langle R_{\mu}, R_{\lambda}\rangle_{\delta_{k}}\\
            &= R_{\mu}(x)e^{-t\sqrt{\theta_{\mu}}}.
  \end{align*}
  Therefore $$\lim_{t\to 0} \partial_{x}^{\alpha}\partial_{t}^{m}u_f(x, t) = (-\sqrt{\theta_{\mu}})^{m}\,\partial_{x}^{\alpha}f(x).$$
  For general $f\in C^{\infty}(\T)$ the situation becomes more intricate. 
We argue as in the proof of the Paley-Wiener theorem \cite[Thm. 8.6]{Op95}. To this end, we consider the generalized Fourier coefficients $\widehat f^{\,k}(\lambda). $
For a polynomial $p\in \mathbb C[\frak a_{\mathbb C}]$ denote by $p(D)$ the associated Cherednik operator. Then  
according to formula  \eqref{Cherednik_FT},  
\begin{align*}
  (p(D)f)^{\wedge k}(\lambda) = p(i\wt\lambda\,) \widehat f^{\,k}(\lambda).
\end{align*}
Taking $p(z)\coloneqq \langle z, z\rangle^{j}$ with $j\in \mathbb N_0$ 
this becomes
\begin{align*}
  (p(D)f)^{\wedge k}(\lambda) = (-1)^j |\widetilde \lambda|^{2j} \widehat f^{\,k}(\lambda).
                                          \end{align*}
So for each $j\in\N_{0}$ there is a constant $C_{f,j}\geq 0$ such that
\begin{align}\label{decay}
  |\wt\lambda\,|^{2j}\,|\widehat f^{\,k}(\lambda)| \leq C_{f, j}. 
\end{align}
Now fix $\alpha \in \mathbb N_0^n, \, m\in \mathbb N_0.$ 
Estimate \eqref{decay} implies that the series
$$g(x, t):=\sum_{\lambda\in\wtlatt}r_{\lambda}\partial_{x}^{\alpha} R_{\lambda}(x)\partial_{t}^{m} e^{-t\sqrt{\theta_{\lambda}}}\,\widehat{f}^{\,k}(\lambda)$$
converges normally on $\a\times [0,\infty]$. Hence we conclude that
\begin{align*}
  \lim_{t\to 0} \partial_{x}^{\alpha}\partial_{t}^{m}u_f(x, t) &= \lim_{t\to 0} \int_{\T} \Bigl(\sum_{\lambda}r_{\lambda}\partial_{t}^{m}e^{-t\sqrt{\theta_{\lambda}}}\, \partial_{x}^{\alpha}R_{\lambda}(x)R_{\lambda}(-y)\Bigr) f(y)\delta_{k}(y) dy\\
  &= \lim_{t\to 0} g(x, t) = g(x, 0) 
\end{align*}
exists. This shows that  $u_f\in C^{\infty}(\T\times [0,\infty])$ and finishes the proof.
\end{proof}

\begin{lemma}\label{Poisson_Cherednik} The Cherednik operators $D_\xi$ commute with the action of the Poisson semigroup $(P(t))_{t\geq 0}$ on $C^1(\T),$ i.e.
$$D_\xi P(t)f = P(t)D_\xi f$$
for all $f\in C^1(\T)$ and $t\geq 0.$
\end{lemma}

\begin{proof} Using \eqref{Cherednik_EF} and \eqref{Cherednik_FT}, we calculate  \begin{align*}
    D_{\xi}P(t)f(x) &=  D_{\xi}\left( \sum_{\lambda\in P} r_{\lambda}e^{-t\sqrt{\theta_{\lambda}}}\, \widehat f^{\,k}(\lambda) R_\lambda\right)(x)\\
                &=\sum_{\lambda\in P} r_{\lambda} e^{-t\sqrt{\theta_{\lambda}}}\,\widehat f^{\, k}(\lambda)\,                 i\langle \widetilde\lambda, \xi\rangle R_{\lambda}(x) \\
                &= \sum_{\lambda\in P}r_{\lambda}e^{-t\sqrt{\theta_{\lambda}}} (D_{\xi} f)^{\wedge k}(\lambda) R_{\lambda}(x)\\
                &=P(t)(D_{\xi}f)(x).
  \end{align*}
	\end{proof}

For $f\in L^{p}(\T,\,\delta_{k})$ with $1\leq p\leq \infty$ we define the maximal function $P^{*}\!f$ associated to the Poisson semigroup by
\begin{align*}
  P^{*}\!f(x)\coloneqq \sup_{t>0}|P(t)f(x)|. 
\end{align*}

\begin{lemma}\label{lem:maxfuncContOp}
  The operator $f\mapsto P^{*}\!f$ is bounded on $L^{p}(\T,\,\delta_{k})$ for $1 < p\leq \infty$ and of weak type $(1, 1)$. 
\end{lemma}
\begin{proof}
  The proof is standard, c.f. \cite{St70}, Lemma 1 on p.48f\,: First, integration by parts in
  \begin{align*}
    P(t)f = \frac{1}{t^{2}}\int_{0}^{\infty}\phi\Big(\frac{s}{t^{2}}\Big)H(s)f\ ds \quad\text{with}\ \phi(s)=\frac{1}{\sqrt{4\pi}}e^{-1/4s}s^{-3/2} 
  \end{align*}
  shows that
  \begin{align*}
    |P(t)f(x)| \leq A\sup_{s>0}\frac{1}{s}\int_{0}^{s}|H(\tau)f(x)|\,d\tau, 
  \end{align*}
  with some constant $A>0$ independent of $f, t$ and $x$. Further, it follows by the Hopf-Dunford-Schwartz ergodic theorem (\cite[Chapt. VIII.7]{DS58}) that the maximal function $M$ defined by
  \begin{align*}
    Mf(x)=\sup_{s>0}\frac{1}{s}\int_{0}^{s} |H(\tau)f(x)|\, d\tau
  \end{align*}
  is bounded on $L^{p}(\T,\,\delta_{k})$ for $1<p\leq \infty$ and of weak type $(1,1)$. This implies the assertion.
\end{proof}

  \section{The Littlewood-Paley $g$-functions}\label{sec:gfun}

The aim of this section is to introduce several Littlewood-Paley $g$-functions in the compact Heckman-Opdam setting and prove their $L^{p}$-boundedness for $1<p\leq 2$. We essentially follow the approach of \cite[Chapt. II]{St70} for compact Lie groups. In our setting, Cherednik operators will take the place of left-invariant vector fields on a compact Lie group. From a structural point of view, we deviate from \cite{St70}, as Cherednik operators are not pure differential operators and contain reflection terms. Their handling requires additional effort. 
In the following, the positive system $R_{+}\subset \a $ with multiplicity $k\geq 0$ is always fixed. 

\begin{definition}\label{grad_def}
We fix an orthonormal basis $(\xi_{j})_{1\leq j\leq n}$ of $\a$ and introduce the following notations for  $h\in C^{1}(\T\times \opint)$: 
the modified Cherednik operators
 $$ \wt D_{\xi} h \coloneqq D_{\xi}h + i\langle \rho, \xi\rangle h \,=\, \partial_\xi h + \sum_{\alpha \in R_+} ik_\alpha \langle\alpha, \xi \rangle\, d_\alpha h\quad (\xi\in \a),$$
 the gradient and the space-time gradient
 $$ \nabla_{\!x} h \coloneqq (\partial_{\xi_{1}}h, \ldots, \partial_{\xi_{n}}h)^{T},\quad 
  \nabla h \coloneqq (\partial_{\xi_{1}}h, \ldots, \partial_{\xi_{n}}h, \partial_{t}h)^{T},$$
 as well as  the space-time Dunkl gradient 
  $$ \wt\nabla_{k}h \coloneqq (\wt D_{\xi_{1}}h, \ldots, \wt D_{\xi_{n}}h, \partial_{t}h)^{T}.$$
Accordingly, we define
\begin{align*}
  |\nabla_{\!x}h|^2 := \sum_{j=1}^{n}|\partial_{\xi_{j}}h|^{2}, \>\> |\nabla h|^2\coloneqq |\partial_{t}h|^{2}+\sum_{j=1}^{n} |\partial_{\xi_{j}}h|^{2},\>\>
  |\wt\nabla_{k} h|^2\coloneqq |\partial_{t}h|^{2}+\sum_{j=1}^{n} |\wt D_{\xi_{j}}h|^{2}. 
\end{align*}
Note that these expressions are independent of the choice of the orthonormal basis $(\xi_j).$
Further, we introduce the space-time Heckman-Opdam Laplacian 
\begin{align*}
  \Delta_{k}\coloneqq L_{k}+\partial_{t}^{2}
\end{align*}
on $C^{2}(\T\times\opint)$; here $L_{k}$ is understood to act with respect to $x\in\T$.

For  $f\in L^{p}(\T,\,\delta_{k})$ we again denote by $u_{f}(x, t)=P(t)f(x)$ the Poisson integral of $f$ and define the following Littlewood-Paley $g$-functions:
\begin{align*}
  g_{\partial_{\xi}}(f)(x) &\coloneqq \left(\int_{0}^{\infty} t\,|\partial_{\xi} u_{f}(x, t)|^{2}\ dt\right)^{1/2}, 
                           &g_{\wt D_{\xi}}(f)(x) \coloneqq \left(\int_{0}^{\infty} t\,|\wt D_{\xi} u_{f}(x, t)|^{2}\ dt\right)^{1/2},\\
  g_{\nabla}(f)(x)&\coloneqq \left(\int_{0}^{\infty}t\,|\nabla u_{f}(x, t)|^{2}\ dt\right)^{1/2},
                  &g_{\wt\nabla_{k}}(f)(x)\coloneqq \left(\int_{0}^{\infty}t\,|\wt\nabla_{k} u_{f}(x, t)|^{2}\ dt\right)^{1/2}.
  \end{align*}
Further,
\begin{align*}
  g_{0}^{\alpha}(f)(x)&\coloneqq \left(\int_{0}^{\infty} t\left\vert\frac{u_{f}(x, t)-u_{f}(s_{\alpha}x, t)}{1-e^{-i\langle \alpha, x\rangle}}\right\vert^{2} dt\right)^{1/2}\,=\,   \left(\int_0^\infty t\,|d_\alpha u_f(x,t)|^2dt\right)^{1/2},  
  \end{align*}
  where reflection operators are understood to act with respect to the variable $x$, and                                                                                                                          
\begin{align*} g_{0}(f)(x)&\coloneqq \max\big\{g_{0}^{\alpha}(f)(x): \alpha\in R_+ \text{ with } k_\alpha >0\big\} \\
    g_{\Delta_{k}}(f)(x)&\coloneqq \left(\int_{0}^{\infty} t\, |\Delta_{k}(u_{f}^{2})(x, t)|\, dt\right)^{1/2}.
  \end{align*}
 \end{definition}
  
 The main result of this section will be the following theorem concerning the $L^p$-boundedness of these $g$-functions.

\begin{theorem}\label{thm:gFunctionBound}
  Denote $\|\cdot\|_{p}\coloneqq\|\cdot\|_{p, \delta_{k}}$ and fix $\xi\in\a$ with $|\xi|\leq 1$.
  \begin{enumerate}
    \item For $1<p\leq 2$ and a subscript $\bullet\in\{0, \partial_{\xi}, \nabla, \Delta_{k}\}$, there exists a constant $A_{\bullet, p}>0$ such that for all $f\in L^{p}(\T, \delta_{k})$,
  \begin{align*}
    \|g_{\bullet}(f)\|_{p} \leq A_{\bullet, p}\|f\|_{p}. 
  \end{align*}

    \item For $1<p<\infty$ and a subscript $\bullet\in\{\wt D_{\xi}, \wt\nabla_{k}\}$, there exists a constant $A_{\bullet, p}>0$ such that the assertion of part (1) holds.
  \end{enumerate}
  \end{theorem}
  The proof of (1) requires several steps and will take up the rest of this section. Part (2) will then be a consequence once we have established Riesz transforms in \cref{sec:riesz}.
We start with some lemmata concerning the space-time Laplacian $\Delta_{k}$.

\begin{lemma}\label{lem:laplacePpowerOfHarmonic}
  Suppose that $v\in C^{2}(\T\times \opint)$ is strictly positive and $\Delta_{k}$-harmonic, i.e. $\Delta_{k} v = 0$. Then for any $p>1$, 
  \begin{align*}
    \Delta_{k}(v^{p})= p(p-1)v^{p-2}|\nabla v|^{2} + pv^{p-1}L_{k}^{\operatorname{ref}}(v)-L_{k}^{\operatorname{ref}}(v^{p}). 
  \end{align*}
\end{lemma}
\begin{proof}
  A short calculation gives
  \begin{align*}
    \partial_{t}^{2}(v^{p})=p(p-1)v^{p-2}(\partial_{t}v)^{2}+pv^{p-1} \partial_{t}^{2}v
  \end{align*}
  and
  \begin{align*}
    L_{k}^{\text{diff}}(v^{p}) &= \Delta(v^{p})+\sum_{\alpha\in R_{+}}k_{\alpha}\cot\frac{\langle \alpha, \cdot\rangle}{2} \partial_{\alpha}(v^{p})\\
                               &= p(p-1)v^{p-2}|\nabla_{x}v|^{2} + pv^{p-1}L_{k}^{\text{diff}}(v).
  \end{align*}
  Thus combining the previous two equations with the harmonicity we obtain
  \begin{align*}
    \Delta_{k}(v^{p})&= (L_{k}^{\text{diff}}+\partial_{t}^{2}-L_{k}^{\text{ref}})(v^{p})\\
                     &=p(p-1)v^{p-2}|\nabla v|^{2} + pv^{p-1}(L_{k}^{\text{diff}}+\partial_{t}^{2})(v)-L_{k}^{\text{ref}}(v^{p})\\
                     &=p(p-1)v^{p-2}|\nabla v|^{2}+ pv^{p-1}(L_{k}+L_{k}^{\text{ref}}+\partial_{t}^{2})(v)-L_{k}^{\text{ref}}(v^{p})\\
                     &= p(p-1)v^{p-2}|\nabla v|^{2} + pv^{p-1}L_{k}^{\text{ref}}(v)-L_{k}^{\text{ref}}(v^{p}).
  \end{align*}
\end{proof}
We note that $L_{k}$ acts on $W$-invariant functions belonging to $C^{2}(\T\times\opint)^{W}$ as a differential operator and has a significantly simpler form on this closed subspace. Thus, if we were only interested in a $W$-invariant setting, this would allow for a proof closer to that in \cite{St70}. We shall come back to the $W$-invariant setting in Section \ref{sec:symmetric}.
In the present situation, however, we have to deal with the reflection terms. For this, we adapt ideas from \cite{S05} and \cite{LZL17}.

\begin{lemma}\label{lem:integraltDelta}
  Let $v\in C^{2}(\T\times [0, \infty])$ and assume that $\,t\partial_{t}v(\cdot, t)\to 0$ uniformly on $\T$ as $t\to 0$ or $t\to\infty$. Then
  \begin{align*}
    \int_{0}^{\infty} \!\int_{\T} t\,\Delta_{k}v(x, t)\delta_{k}(x) \,dx \,dt \,= \int_{\T} (v(x, 0)-v(x, \infty)) \delta_{k}(x)\, dx. 
  \end{align*}
  In particular, the integral on the left side converges.
\end{lemma}
\begin{proof}
Consider $f\in C^{2}(\T)$. By the anti-symmetry of the Cherednik operators $D_{\xi}$, we have 
  \begin{align*}
    \int_{\T} (D_{\xi} f)\,\delta_{k}\, dx \,=\,-\int_{\T}f\cdot \ol{D_{\xi}1}\,\delta_{k}\ dx \,=\, -i\langle\rho, \xi\rangle \int_{\T} f\,\delta_{k}\, dx 
  \end{align*}
  and
  \begin{align*}
    \int_{\T}(D_{\xi}^{2}f)\,\delta_{k}\, dx \,=\, -\langle \rho, \xi\rangle^{2}\int_{\T}f\,\delta_{k}\, dx. 
  \end{align*}
  Thus for $f\in C^{2}(\T)$ we obtain, with an arbitrary orthonormal basis $(\xi_{j})$ of $\a$,
  \begin{align}\label{eqn:LkIntegralZero}
    \int_{\T}(L_{k}f)\, \delta_{k}\, dx\, =\,\int_{\T}\Bigl(\sum_{j=1}^{n}D_{\xi_{j}}^{2}+|\rho|^{2}\Bigr)(f)\, \delta_{k}\, dx\, = 0. 
  \end{align}
  We define $h\in C^{2}([0,\infty])$ as the  integral
  \begin{align*}
    h(t)\coloneqq \int_{\T}v(x, t)\,\delta_{k}(x)\, dx.
  \end{align*}
  Then by our assumptions, $\,th'(t)\to 0$ as $t\to 0$ or $t\to\infty$.
  From \eqref{eqn:LkIntegralZero} and integration by parts,  we get
  \begin{align*}
    \int_{0}^{\infty}t\,\Delta_{k}v(x, t)\,\delta_{k}(x)\, dx\, dt =& \int_{0}^{\infty}t \int_{\T}\partial^{2}_{t}v(x, t)\,\delta_{k}(x)\, dx\,dt \\
    =&\int_{0}^{\infty}th''(t)\, dt\\
    =&\big[th'(t)\big]_{0}^{\infty}-\int_{0}^{\infty} h'(t)\, dt\\
    =&\ h(0)-h(\infty).
  \end{align*}\end{proof}

As the following result shows, our various $g$-functions are closely related (compare this with Lemma 2.1. of \cite{LZL17}).

\begin{proposition}\label{prop:gineq}
  Let $c := \sum_{\alpha\in R_+}k_\alpha|\alpha|. $  Then for $\xi\in \frak a$ with $|\xi|\leq 1$ and $f\in L^{p}(\T, \delta_{k})$ with $1\leq p<\infty$, 
    \begin{enumerate}
    \item $g_{\wt D_{\xi}}(f)\leq g_{\partial_{\xi}}(f)+cg_{0}(f)$ and $g_{\partial_{\xi}}(f)\leq g_{\wt D_{\xi}}(f)+ c g_{0}(f)$.
    \item $g_{\wt \nabla_{k}}(f) \leq g_{\nabla}(f)+cg_{0}(f)$ and $g_{\nabla}(f)\leq g_{\wt\nabla_{k}}(f)+cg_{0}(f)$.
    \item If $f$ is real-valued, then $g_{\Delta_{k}}(f)^{2}= 2g_{\nabla}(f)^{2}+\sum_{\alpha\in R_{+}}k_{\alpha}|\alpha|^{2}g^{\alpha}_{0}(f)^{2}$.
  \end{enumerate}
  \end{proposition}
	\begin{proof}
  We consider the Hilbert space $\H:=L^{2}(\opint, \,t\,dt)$ with its $L^{2}$-norm $\|\cdot\|_{\H}$ and inner product $\langle \cdot, \cdot\rangle_{\H}$. For measurable functions $g:\,]0,\infty[\to \mathbb C$, we use the notation
  $$ \|g\|_\mathcal H := \Bigl(\int_0^\infty t|g(t)|^2dt\Bigr)^{1/2}$$
  also if this integral is infinite. 
  Recalling Definition \ref{grad_def}, we further write $$\widetilde D_{\xi}=\partial_{\xi}+D_{\xi}^{\operatorname{ref}} \quad \text{ with } \quad  D_{\xi}^{\text{ref}} \coloneqq \sum_{\alpha\in R_{+}}ik_{\alpha}\langle \alpha, \xi\rangle d_\alpha$$ 
  and put $u(x, t)\coloneqq u_f(x,t)=P(t)f(x)$. With these notations, we have
    $$\| \partial_\xi u(x,.)\|_{\H}\, = \, g_{\partial_\xi}(f)(x) $$
    and 
 $$  \| D_\xi^{\text{ref}} u(x,.)\|_{\H} \,\leq \, \sum_{\alpha\in R_+} k_\alpha |\langle \alpha, \xi \rangle |\, g_0^\alpha(f)(x) \, \leq \, cg_0(f)(x).$$
  Hence
$$ g_{\wt D_{\xi}}(f)(x) = \|(\partial_{\xi}+D_{\xi}^{\text{ref}}\,)u(x, .)||_{\H} \, \leq  g_{\partial_\xi}(f)(x) + \,c g_0(f)(x). $$
   The second inequality in (1) is proven in the same way.
  We further calculate
  \begin{align*}
    |\wt\nabla_{k}u|^{2} = |\nabla u|^{2}+2\sum_{\alpha\in R_{+}} k_{\alpha} \,\Im\bigg(\sum_{j=1}^{n} \langle \alpha, \xi_{j}\rangle\,\partial_{\xi_{j}}u\ \overline{d_{\alpha}u}\bigg)+\sum_{\alpha, \beta\in R_{+}} k_{\alpha}k_{\beta}\langle \alpha,\beta\rangle d_{\alpha}u \cdot \overline{d_{\beta}u}. 
  \end{align*}
  It follows that 
  \begin{equation}\label{widetilde_Abschaetzung} |\wt\nabla_k u| \,\leq |\nabla u| + \,c'\cdot\sum_{\alpha \in R_+: k_\alpha>0} |d_\alpha u|\end{equation}
  and therefore 
  \begin{align*} g_{\wt\nabla_k}(f)(x) &= \|\wt\nabla_ku(x,.)\|_\H \,\leq \, \|\nabla u(x, .)\|_\H + \,c\cdot \max_{\alpha\in R_+:k_\alpha>0} \|d_\alpha u(x,.)\|_\H \\
  &=\, 	g_{\nabla}(f)(x) + c \,g_0(f)(x).
  \end{align*}
The second inequality of (2) is proven in the same way. For part (3) observe first that $u$ is real-valued since $f$ is real-valued. Thus 
  \begin{align*}
    L_{k}(u^{2}) &= \Delta(u^{2}) + \sum_{\alpha\in R_{+}} k_{\alpha}\cot\frac{\langle \alpha, \cdot\rangle }{2}\partial_{\alpha}(u^{2})-\sum_{\alpha\in R_{+}}k_{\alpha}\frac{|\alpha|^{2}}{4\sin^{2}\frac{\langle \alpha, \cdot\rangle }{2}}(u^{2}-s_\alpha u^{2})\\
                 &=2uL_{k}u+2|\nabla_{x}u|^{2}+\sum_{\alpha\in R_{+}} k_{\alpha}\frac{|\alpha|^{2}}{4\sin^{2}\frac{\langle \alpha,\cdot\rangle }{2}}(u-s_\alpha u)^{2},
  \end{align*}
  because $(u-s_{\alpha}u)^{2}-2u (u-s_{\alpha}u)=(s_{\alpha}u)^{2}-u^{2}$. Since $(L_{k}+\partial_{t}^{2})u = 0$, we conclude
  \begin{align*}
    \Delta_{k}(u^{2})=(L_{k}+\partial_{t}^{2})(u^{2})=2|\nabla u|^{2} + \sum_{\alpha\in R_{+}}k_{\alpha}\frac{|\alpha|^{2}}{4\sin^{2}\frac{\langle \alpha, \cdot\rangle }{2}} \vert u-s_\alpha u\vert^{2}. 
  \end{align*}
  Upon noting that $\,4\sin^{2}\frac{\langle \alpha, x\rangle}{2} = |1-e^{-i\langle \alpha, x\rangle}|^{2}$, we obtain (3).
\end{proof}

In the next proposition, we prove $L^{p}$-boundedness of $g_{0}$ and $g_{\nabla}$ in the case $1<p\leq 2$. 

\begin{proposition}\label{prop:gNablaBoundPLess2}
  For $1<p\leq 2$ and $\bullet\in\{0, \nabla\}$ there exists a constant $A_{p}>0$, depending on $R$ and $k$, such that for all $f\in L^{p}(\T,\delta_{k}),$
  \begin{align*}
    \|g_{\bullet}(f)\|_{p} \leq A_{p}\|f\|_{p}\,. 
  \end{align*}
 Here again $\|.\|_p = \|.\|_{p, \delta_k}$.
  \end{proposition}
\begin{proof} We essentially follow the proof of \cite[p. 50f]{St70}, but with a slight variation similar to \cite{S05} and \cite{LZL17}. 
It suffices to consider $f\in C^{\infty}(\T)$ with $f\geq\ve>0$, so that $u\in C^{\infty}(\T\times [0,\infty])$ and $u(x, t)\geq\ve >0$ by the strict positivity of the Poisson kernel. We define $F\in C^{\infty}(\T\times [0,\infty])$ by $$F(x, t):= u(x, t)^{p}.$$ Note that $t\partial_{t}F(\cdot, t)\to 0$ uniformly on $\T$ as $t\to 0$ or $t\to\infty$ as a consequence of \cref{lem:smoothAtTimeBoundary}.
Then put $$U_{p}\coloneqq U_{p}^{(1)}+U_{p}^{(2)},$$ where 
\begin{align*}
  U_{p}^{(1)} &\coloneqq \frac{p}{2}\sum_{\alpha\in R_{+}}k_{\alpha}|\alpha|^{2}\frac{(u^{p-1}-s_{\alpha} u^{p-1})(u-s_{\alpha}u)}{4\sin^{2}\frac{\langle \alpha, \cdot\rangle}{2}},
\end{align*}
and
\begin{align*}
  U_{p}^{(2)} &\coloneqq \frac{p}{2}\sum_{\alpha\in R_{+}}k_{\alpha}|\alpha|^{2}\frac{(u^{p-1}+s_{\alpha}u^{p-1})(u-s_{\alpha}u)}{4\sin^{2}\frac{\langle \alpha,\cdot\rangle }{2}} - \sum_{\alpha\in R_{+}}k_{\alpha}|\alpha|^{2}\frac{u^{p}-s_{\alpha}u^{p}}{4\sin^{2}\frac{\langle \alpha, \cdot\rangle}{2}}.
\end{align*}
Note that in fact
\begin{align*}
  U_{p} = pu^{p-1}L_{k}^{\operatorname{ref}}(u)-L_{k}^{\operatorname{ref}}(u^{p}),
\end{align*}
as we just regrouped some terms.
Furthermore, observe that $U_{p}^{(1)}\geq 0$ as the expressions $u^{p-1}-s_{\alpha}u^{p-1}$ and $u-s_{\alpha}u$ have the same sign.

According to \cref{lem:laplacePpowerOfHarmonic}, 
\begin{align}\label{eqn:deltaU1U2split}
  \Delta_{k}F-U_{p}^{(2)}= p(p-1)u^{p-2}|\nabla u|^{2} + U_{p}^{(1)} \geq 0.
\end{align}
In the following integrations over $\mathbb T,$ we shall need to approximate the domain of integration in order to avoid singularities. A dense open subset of $\mathbb T$ is given by $W.\alc^{\circ}$ with the open
fundamental alcove 
$\,\alc^{\circ}\coloneqq\{x\in\a : 0< \langle \alpha, x\rangle < 2\pi\ \> \forall \alpha\in R_{+}\}.$  
In this spirit, we consider $\Omega_{N}:=W.A_{N}$ for $N\in \mathbb N,$ where
\begin{align*}
  A_{N}\coloneqq \{x\in \a : \tfrac 1 N<\langle \alpha, x\rangle <2\pi-\tfrac 1 N\ \text{for all}\ \alpha\in R_{+}\}.
\end{align*}
Then
\begin{align*}
  \int_{\Omega_{N}} L_{k}^{\text{ref}}(u^{p})\,\delta_{k}\, dx\, &=\, \sum_{\alpha\in R_{+}} k_{\alpha}|\alpha|^{2}\int_{\Omega_{N}} \frac{u^{p}-s_{\alpha}(u^{p})}{4\sin^{2}\frac{\langle \alpha,\cdot\rangle }{2}}\,\delta_{k}\, dx = \,0,
\end{align*}
as $\Omega_N$ is $W$-invariant and the integrand is odd with respect to $x\mapsto s_{\alpha}x$. Similarly,
\begin{align*}
  \int_{\Omega_{N}} \frac{(u^{p-1}+s_{\alpha} u^{p-1})(u-s_{\alpha}u)}{4\sin^{2}\frac{\langle \alpha, \cdot\rangle}{2}}\, \delta_{k}\, dx \, = 0.
\end{align*}
Consequently, 
\begin{align*}
  \int_{\Omega_{N}} U^{(2)}_{p}(x, t)\,\delta_{k}(x)\ dx = 0. 
\end{align*}
Indeed, the rearranging of terms $U_{p}=U_{p}^{(1)}+U_{p}^{(2)}$ was precisely done in this fashion so that $U_{p}^{(2)}$ contains terms invisible for integrals over $W$-invariant domains and $U_{p}^{(1)}$ is non-negative.

We write $f^{*} = P^{*}\!f$ for the maximal function of $f$.  Using \cref{eqn:deltaU1U2split} we calculate
\begin{align*}
  (g_{\nabla}(f)(x))^{2} &= \,\int_{0}^{\infty} t|\nabla u(x, t)|^{2}\, dt\\
                         &\leq \,\frac{1}{p(p-1)} \int_{0}^{\infty} t u(x, t)^{2-p}(\Delta_{k} u^{p}(x, t)-U^{(2)}_{p}(x, t))\, dt\\
                         &\leq \,A_{p} (f^{*}(x))^{2-p} \int_{0}^{\infty} t\,(\Delta_{k}F(x, t)-U^{(2)}_{p}(x, t))\, dt
\end{align*}
with $\, A_p = \frac{1}{p(p-1)}, $ because $u(x, t)\leq f^{*}(x)$. 

Now put 
$$I(x)\coloneqq \int_{0}^{\infty}t\,(\Delta_{k}F(x, t)-U^{(2)}_{p}(x, t))\ dt \geq 0.$$
In consequence of \cref{lem:integraltDelta}, we obtain
\begin{align*}
  \int_{\T} I(x)\,\delta_{k}(x)\, dx \,&=\,\int_{0}^{\infty} \int_{\T} t\,(\Delta_{k}F - U^{(2)}_{p})\, \delta_{k}\, dx\, dt\\
                                       &=\,\lim_{N\to\infty} \int_{0}^{\infty}\int_{\Omega_{N}} t\,\Delta_{k}F\,\delta_{k}\, dx\, dt - \lim_{N\to\infty}\int_{0}^{\infty}t\int_{\Omega_{N}} U^{(2)}_{p}\, \delta_{k}\, dx\, dt\\
                                    &\leq\, \int_{0}^{\infty}\int_{\T}t\,\Delta_{k}F\,\delta_{k}\, dx\, dt\\
                                   &=\,\int_{\T} F(x, 0)\,\delta_{k}(x)\ dx - \int_{\T} F(x, \infty)\,\delta_{k}(x)\, dx\\
                                   &\leq\, \int_{\T} F(x, 0)\,\delta_{k}(x)\, dx = \|f\|_{p}^{p}.
\end{align*}
Here the change of the order of integration in the first equality  is justified by the nonnegativity of the integrand. 
Thus by H\"older's inequality and  \cref{lem:maxfuncContOp}, 
\begin{align}\label{ineq:gnablaCalculation}
  \|g_{\nabla}(f)\|^{p}_{p} &\leq A'_{p} \int_{\T} (f^{*}(x))^{(2-p)p/2} I(x)^{p/2}\,\delta_{k}(x)\, dx\nonumber\\
                                     &\leq A'_{p} \left(\int_{\T} |f^{*}(x)|^{p}\,\delta_{k}(x)\, dx\right)^{(2-p)/2} \left(\int_{\T} I(x)\,\delta_{k}(x)\, dx\right)^{p/2}\nonumber\\
                                     &\leq A_{p}'' \,\|f\|_{p}^{p(2-p)/2}\|f\|_{p}^{p^{2}/2} = A''_{p} \|f\|_{p}^{p}. 
\end{align}
So $g_{\nabla}$ is $L^{p}$-continuous for $1<p\leq 2$. 

For $g_{0}$ we proceed in a similar fashion: 
In view of \cref{eqn:deltaU1U2split} it suffices to prove that for each $\alpha\in R_{+}$,
\begin{align*}
  |d_{\alpha}u(x, t)|^{2}\leq A_{p} f^{**}(x)^{2-p}\, U_{p}^{(1)}(x, t)
\end{align*}
with some function $f^{**}$ satisfying $\|f^{**}\|_{p}\leq C_{p}\|f\|_{p}$. Indeed, we may then employ \cref{eqn:deltaU1U2split} and perform the same calculation as in \cref{ineq:gnablaCalculation} with $g_{0}$ in place of $g_{\nabla}$, since
\begin{align*}
  U_{p}^{(1)}\leq \Delta_{k}F-U_{p}^{(2)}. 
\end{align*}

But for $1<p\leq 2$ and $a, b\geq 0$ we have
\begin{align*}
  |a^{p-1}-b^{p-1}| \geq (p-1)\max\{a, b\}^{p-2}|a-b| 
\end{align*}
by integrating the function $s\mapsto s^{p-2}$. Thus in particular,
\begin{align*}
  (u^{p-1}(x, t)-u^{p-1}(s_{\alpha}x, t))(u(x, t)-u(s_{\alpha}x, t)) \geq (p-1) (\max_{w\in W} f^{*}(wx))^{p-2} (u(x, t)-u(s_{\alpha}x, t))^{2}. 
\end{align*}
The function $f^{**}(x)\coloneqq \max_{w\in W} f^{*}(wx)$ obviously satisfies the desired properties. This finishes the proof.
\end{proof}

By virtue of \cref{prop:gineq}, this result now implies part (1) of \cref{thm:gFunctionBound}.

Although we shall not need the next Lemma, it exemplifies how the modified gradient $\wt\nabla_{k}$ is adapted to the structure of the Poisson semigroup. The reader should compare this with \cite[p. 53]{St70}.
\begin{lemma}[Subharmonicity]\label{lem:gNablaKSubharmonic}
  Let $f\in C^{\infty}(\T)$ and $u(x, t)=P(t)f(x)$. Then
  \begin{align*}
    |\wt\nabla_{k} u(x, t)|^{2} \leq P(\tfrac t 2)(|\wt\nabla_{k} u(\cdot, \tfrac t 2)|^{2})(x).
  \end{align*}
\end{lemma}
\begin{proof}
By Lemma \ref{Poisson_Cherednik}, we have 
$$ \widetilde D_{\xi_j} u(\cdot, t)= P(\tfrac{t}{2})\widetilde D_{\xi_j} u(\cdot, \tfrac{t}{2}) $$
and thus by H\"older's inequality,
 \begin{align*}
    |\wt D_{\xi_{j}} u(x, t)| 
                            &= \left\vert\int_{\T} P_{k}(\tfrac t 2, x, y) \wt D_{\xi_{j}}u(y, \tfrac t 2)\,\delta_{k}(y)\, dy\right\vert\\
                            &\leq \left(\int_{\T} P_{k}(\tfrac t 2, x, y)\,\delta_{k}(y)\, dy\right)^{1/2} \left(\int_{\T}|\wt D_{\xi_{j}}u(y, \tfrac t 2)|^{2}P_{k}(\tfrac t 2, x, y)\,\delta_{k}(y)\, dy\right)^{1/2}\\
                                &= \Big( P(\tfrac t 2)(|\wt D_{\xi_{j}}u(\cdot,\tfrac t 2)|^{2})(x)\Big)^{1/2}.
  \end{align*}
 Further, using the expansion
 $$ u(x,t) = \sum_{\lambda\in P} r_\lambda e^{-t\sqrt{\theta_\lambda}} R_\lambda(x)\langle f, R_\lambda\rangle_{\delta_k} ,$$
 one easily checks that
 $\, \partial_tu(x,s_1+s_2) = P(s_1) \partial_tu(x,s_2).\,$
 In particular, 
 $$\partial_tu(x,t) = P(\tfrac{t}{2})(\partial_tu)(\cdot, \tfrac{t}{2}) $$
 This leads to an estimate for $\partial_{t}u(x,t)$ analogous to that for $\widetilde D_{\xi_j}u(x,t)$ above,  which concludes the proof.
\end{proof}

Recall that the Poisson semigroup $(P(t))_{t\geq 0}$ is a symmetric diffusion semigroup, i.e. it is positive and contractive on $L^{p}(\T,\delta_{k})$ for $1 \leq p \leq \infty$ with $P(t)1=1$, strongly continous for $1\leq p<\infty,$  and each $P(t)$ is self-adjoint on $L^{2}(\T, \delta_{k})$. We may therefore employ general Littlewood-Paley theory based on symmetric diffusion semigroups from \cite[Ch.IV]{St70} to obtain the following converse of \cref{thm:gFunctionBound} for $g_{\nabla}$ and $g_{\widetilde\nabla_k}$.

\begin{theorem}\label{thm:g1bound}
  Let $1<p<\infty$ and $\,\bullet\in\{\nabla, \wt\nabla_{k}\}$. Define for $f\in L^{p}(\T,\, \delta_{k})$ the  $g_{1}$-function
  \begin{align*}
    g_{1}(f)(x)\coloneqq \left(\int_{0}^{\infty}t\, |\partial_{t}u(x, t)|^{2}\ dt\right)^{1/2}, 
  \end{align*}
  where $u(x, t)=P(t)f(x)$ as before.  Suppose further that
   $\displaystyle \int_{\T}f(x)\,\delta_{k}(x)\,dx = 0.$ Then $$\|f\|_{p}\leq B_{p}\|g_{1}(f)\|_{p}\leq B_{p}\|g_{\bullet}(f)\|_{p}$$ with some constant $B_{p}\geq 0$ depending on $R$ and $k$ only.
\end{theorem}

\begin{proof}
  This is Corollary 2 of \cite[Ch.IV, Section 6]{St70}, because the constants are the only functions from $L^2(\T, \delta_k)$ which are left invariant by all $P(t), t>0.$  One could alternatively prove the result directly as in \cite[pp. 55ff.]{St70}.
  \end{proof}

In fact, general theory tells us $L^{p}$-boundedness of several Littlewood-Paley $g$-functions involving only time derivatives:

\begin{theorem}\label{thm:gellbounds}
  Let $1<p<\infty$ and $\ell\in\N$. We define for $f\in L^{p}(\T, \delta_{k})$ the  $g$-functions
  \begin{align*}
    g_{\ell}(f)(x)\coloneqq \left(\int_{0}^{\infty} t^{2\ell-1}|\partial_{t}^{\ell}u(x, t)|^{2}\ dt\right)^{1/2},
  \end{align*}
  again with $u(x, t)= P(t)f(x)$. Then there exist constants $A_{\ell, p}\geq 0$, depending on $R$ and $k$, such that $\|g_{\ell}(f)\|_{p}\leq A_{\ell, p}\|f\|_{p}$.
\end{theorem}
\begin{proof}
  This is Corollary 1 of \cite[Ch.IV, Section 6]{St70}. 
\end{proof}

\section{Riesz transforms and imaginary powers}\label{sec:riesz}

In this final section, we introduce natural Riesz transforms and imaginary powers associated with the Heckman-Opdam Laplacian $L_k$ on $\mathbb T.$  We apply our results for $g$-functions to prove their $L^p$-boundedness for $1<p<\infty,$ and we finally finish the proof of   \cref{thm:gFunctionBound}.

We start with the definition of Riesz transforms on the space of trigonometric polynomials $\mathcal T = \text{span}_{\mathbb C}\{e^{i\lambda}, \, \lambda \in P\},$ c.f. \eqref{trigpol}. Note that $\mathcal T$ is dense in $L^p(\T, \delta_k)$ for $1\leq p <\infty.$ 
Note also that in view of Lemma \ref{ev}, the operator $-L_k$ is positive on $L^2(\T, \delta_k)$ with $L_k1=0.$ For an exponent $s\in \mathbb C$, the fractional power $(-L_k)^s$ is defined on $\mathcal T$  by
$$ 
  (-L_{k})^sf = \sum_{\lambda\in P\setminus\{0\}} a_{\lambda} \theta_{\lambda}^s\,R_{\lambda}\in\mathcal T \quad \text{ for } f=\sum_{\lambda\in P} a_{\lambda}R_{\lambda}\in \mathcal T.
  $$

\begin{definition}
  Let $(\xi_{j})_{1\leq j\leq n}$ be an orthonormal basis of $\a$. We define the \emph{$j$-th Riesz transform} of $f\in \mathcal T$ with respect to $(\xi_{j})_{1\leq j\leq n}$ by
  \begin{align*}
    \mathcal R_{j} f \coloneqq \wt D_{\xi_{j}} (-L_{k})^{-1/2}f = (D_{\xi_{j}}+i\langle \rho, \xi_{j}\rangle)(-L_{k})^{-1/2}f.
  \end{align*}  
\end{definition}
For $f=\sum_{\lambda\in P} a_{\lambda}R_{\lambda} \in \mathcal T$ we see from \eqref{Cherednik_EF} that
\begin{align*}
  \mathcal R_{j}f = \sum_{\lambda\in P\setminus\{0\}} \frac{a_{\lambda}}{\sqrt{\theta_{\lambda}}} \,i\langle \wt\lambda+\rho, \xi_{j}\rangle R_{\lambda}.
\end{align*}
Thus, $\mathcal R_j$ is a multiplier operator satisfying
$$(\mathcal R_j f)^{\wedge k} (\lambda) = \frac{i\langle  \widetilde \lambda + \rho, \xi_j\rangle}{\sqrt{\theta_\lambda}} \widehat f^{\,k}(\lambda), \quad \lambda \in P \setminus \{0\}, \quad (\mathcal R_j f)^{\wedge k} (0) = 0.$$

\begin{theorem} \label{Riesz_cont}
 For $1<p<\infty$, the Riesz transforms $\mathcal R_{j}$ extend to bounded linear operators on $L^{p}(\T,\delta_{k}). $
\end{theorem}
\begin{proof}
  We proceed as in \cite{St70}.
  Let $f=\sum_{\lambda\in P}a_{\lambda}R_{\lambda}\in\mathcal T$ and $f_{j}=\mathcal R_{j}f$. Further, let $u(\,.\, , t)=P(t)f$ and $u_{j}(\,.\,, t)=P(t)f_{j}$. Then
  \begin{align*}
    \partial_{t} u(\,.\,,t)  &= \partial_{t} \Bigl(\sum_{\lambda\in P} e^{-t\sqrt{\theta_{\lambda}}} a_{\lambda}R_{\lambda} \Bigr)\\
                   &=\sum_{\lambda\in P\setminus\{0\}} -\sqrt{\theta_{\lambda}}e^{-t\sqrt{\theta_{\lambda}}}a_{\lambda}R_{\lambda} \\
                   &= -P(t)((-L_{k})^{1/2}f).
  \end{align*}
  Applying $\mathcal R_{j}$ to both sides and recalling Lemma \ref{Poisson_Cherednik},     we obtain
  \begin{align}\label{eqn:rieszCherednikVsTime}
    \partial_{t} u_{j} =  \mathcal R_j \partial_tu = -P(t) \wt D_{\xi_j} f = -\wt D_{\xi_{j}}u.
  \end{align}
  Thus we conclude
  \begin{align*}
    g_{1}(\mathcal R_{j}f)(x)= \left(\int_{0}^{\infty} t\, |\partial_{t}u_{j}(x, t)|^{2}\ dt\right)^{1/2} \leq K\left(\int_{0}^{\infty} t\,|\wt \nabla_{k} u(x, t)|^{2}\ dt\right)^{1/2} =K g_{\wt\nabla_{k}}(f)(x).
  \end{align*}
 In view of \cref{thm:g1bound} and \cref{thm:gFunctionBound}(1), this gives for $1<p\leq 2$
 $$\|\mathcal R_{j}f\|_{p}\leq B_{p}\|g_{1}(\mathcal R_{j}f)\|_{p}\leq B^{\prime}_{p}\|f\|_{p}\,.$$ 
 For the case $2\leq p <\infty$ we argue by duality. Let $1<q\leq 2$ denote the conjugate exponent such that $1/p+1/q =1$. Let $f, g\in\mathcal T$. Then
 \begin{align*}
   \langle \mathcal R_{j} f, g\rangle_{\delta_{k}} = -\langle f,\mathcal R_{j} g\rangle_{\delta_{k}}
 \end{align*}
 and thus
 \begin{align*}
   |\langle \mathcal R_{j}f, g\rangle_{\delta_{k}}| \leq \|f\|_{p}\|\mathcal R_{j}g\|_{q}\leq B_{q}'\|f\|_{p}\|g\|_{q}. 
 \end{align*}
 Since $\mathcal T\subseteq L^{q}(\T, \delta_{k})$ is dense, we obtain $\|\mathcal R_{j} f\|_{p}\leq B_{q}' \|f\|_{p}$ for all $f\in\mathcal T$.
 \end{proof}

This result now allows us to finally prove part (2) of \cref{thm:gFunctionBound}. Also compare \cite{LZL17}, Lemma 2.4.
\begin{corollary}
  Let $1<p<\infty$ and $f\in L^{p}(\T, \delta_{k})$. Then
  \begin{enumerate}
    \item $g_{\wt D_{\xi_{j}}}(f) = g_{1}(\mathcal R_{j}f)$.
    \item $g_{\wt\nabla_{k}}(f)^{2} = g_{1}(f)^{2}+\sum_{j=1}^{n} g_{\wt D_{\xi_{j}}}(f)^{2}$.
  \end{enumerate}
  As a consequence, we obtain
   part (2) of \cref{thm:gFunctionBound}.
\end{corollary}
\begin{proof}
  This is  immediate from \cref{eqn:rieszCherednikVsTime}.
\end{proof}

Our Riesz transforms may be used to prove inequalities involving Cherednik operators: 

\begin{corollary}[Riesz inequalities]
  Let $1<p<\infty$ and $f \in C^{2}(\T)$. Then
  \begin{align*}
    \|\wt D_{\xi_{i}}\wt D_{\xi_{j}} f\|_{p}\leq A_{p}\|L_{k}f\|_{p} 
  \end{align*}
  with a constant $A_{p}\geq 0$ depending on $R$ and $k$ only.
\end{corollary}
\begin{proof}
  First consider $f\in \mathcal T$ and let $g=-L_{k}f=(-L_{k})^{1/2}(-L_{k})^{1/2}f$. Then by Theorem \ref{Riesz_cont},
  \begin{align*}
    \|\mathcal R_{i}\mathcal R_{j}g\|_{p}\leq A_{p} \|L_{k}f\|_{p}\,. 
  \end{align*}
  As the $D_{\xi_{j}}$ commute with $(-L_{k})^{1/2}$ on $\mathcal T$, this implies the stated estimate for $f\in \mathcal T.$  Further, the Cherednik operators $D_{\xi}\colon C^{m}(\T)\to C^{m-1}(\T)$ are continuous and thus $L_{k}=\sum_{j=1}^{n}D_{\xi_{j}}^{2}+|\rho|^{2}$ is $C^{2}(\T)\to C^{0}(\T)$ continuous.
  Since $\mathcal T \subseteq C^{2}(\T)$ is dense and    $\|\cdot\|_{\infty}$-convergence implies $\|\cdot\|_{p}$-convergence, this concludes the proof.
\end{proof}

For $\gamma\in \mathbb R,$ we now consider the imaginary powers $(-L_{k})^{i\gamma}$ on $\mathcal T$. 
 
\begin{theorem}\label{thm:rieszPotential}
 For $1<p<\infty, $  the operators $(-L_{k})^{i\gamma},\,\gamma\in\R$ extend to bounded linear operators $L^{p}(\T,\, \delta_{k})\to L^{p}(\T,\delta_{k}).$ 
\end{theorem}
\begin{proof}
  This is a consequence of  \cref{thm:rieszPotentialGeneralized} below, upon noting that
  \begin{align*}
    \theta^{i\gamma} = K\theta \int_{0}^{\infty} e^{-\theta s}s^{-i\gamma}\ ds\qquad (\theta >0)
  \end{align*}
  is a function of Laplace-transform type in the sense of the following definition.
\end{proof}

\begin{definition} (\cite[Ch.II]{St70}) 
  A function $m\colon \opint \to \C$ is said to be of \emph{Laplace-transform type}, if there exists a function $a\in L^{\infty}(\opint)$ such that
  \begin{align*}
    m(\theta) = \theta\int_{0}^{\infty}e^{-\theta s}a(s)\ ds. 
  \end{align*}
\end{definition}

\begin{theorem}\label{thm:rieszPotentialGeneralized}
  Let $m\colon\opint\to\C$ be of Laplace-transform type. Define the operator $I_m$ on $\mathcal T$ by 
  \begin{align*}
    I_m(f) \coloneqq \sum_{\lambda\in P\setminus\{0\}} m(\theta_{\lambda}^{1/2}) a_{\lambda}R_{\lambda}\quad\text{for }\quad f=\sum_{\lambda\in P} a_{\lambda}R_{\lambda}\in \mathcal T. 
  \end{align*}
  Then for $1<p<\infty$, there is a constant $A_p(m)\geq 0$ such that $\|I_{m}(f)\|_{p} \leq A_{p}(m) \|f\|_{p}$ for all $f\in \mathcal T.$ 
   
\end{theorem}

With  $m(\theta)=\theta^{2i\gamma}$ we obtain \cref{thm:rieszPotential}.
\begin{proof}
  The proof is identical to that in \cite{St70}, pp. 59ff. In fact, it only relies on the $L^{p}$-bounds for the $g_{\ell}$ functions established in \cref{thm:gellbounds}. Alternatively, this theorem is also a direct consequence of the general theory on symmetric diffusion semigroups, namely Corollary 3 of \cite[Ch. IV, Sec. 6]{St70}.
\end{proof}

\section{The $W$-invariant (symmetric) case}\label{sec:symmetric}

We now restrict ourselves to the  $W$-invariant setting as considered in \cite{RR11} (up to a different normalization of the root system and the multiplicity). For a space $M$ of functions on $\mathbb T,$ we denote by $M^W$ the subspace of $W$-invariant elements. Note that for $f\in C^\infty(\mathbb T)^W,$  
$$ \widetilde D_\xi f = \partial_\xi f, \quad \widetilde \nabla_k f = \nabla f, \quad L_kf = L_k^{\text{diff}} f.$$
This significantly simplifies the analysis.  
Let us become more precise. 
The topological quotient $\T / W$ is homeomorphic to the closed alcove
\begin{align*}
  \alc = \{x\in\a : 0\leq \langle \alpha, x\rangle \leq 2\pi \> \forall \alpha\in R_{+}\}.
\end{align*}
Thus the space $C(\T)^{W}$ naturally identifies with $C(\alc)$. Similarly, $L^{p}(\T, \delta_{k})^{W}\cong L^{p}(\alc, \delta_{k})$. 
The analysis on $\alc$ with respect to the weight $\delta_{k}$ is governed by the symmetric Heckman-Opdam polynomials $(P_{\lambda})_{\lambda\in\domwt}$ indexed by the dominant weights $\domwt=P\cap\weylcham$, recall Remark \ref{rem:sahi}.   
Renormalizing the symmetric Heckman-Opdam polynomials according to
\begin{align*}
  R^{W}_{\lambda} \coloneqq \frac{1}{P_{\lambda}(0)}P_{\lambda}
  \end{align*}
gives an orthogonal basis $(R_{\lambda}^{W})_{\lambda\in\domwt}$ of $L^{2}(\alc, \delta_{k})$ made up by polynomials satisfying
\begin{align*}
  L_{k}^{\operatorname{diff}}R^{W}_{\lambda} = L_{k}R^{W}_{\lambda} = -\theta_{\lambda}R_{\lambda}^{W}. 
\end{align*}
For the last identity, it was used that the operator $L_{k}$ commutes with the action of $W$ on the space $\mathcal T$ of trigonometric polynomials. This can be seen by direct calculation, also see \cite[Cor. 5.1(1)]{Sch08}. In particular, $L_{k}(w.R_{\lambda})=w.(L_{k}R_{\lambda})=-\theta_{\lambda}w.R_{\lambda}$.
It is also immediate from the definitions that 
$$ R_\lambda^W =  \frac{1}{|W|} \sum_{w\in W} w.R_\lambda.$$ 

The symmetric Heckman-Opdam heat kernel studied in \cite{RR11} is given by
\begin{align*}
  \Gamma_{k}^{W}(t, x, y) &= \sum_{\lambda\in\domwt} r_{\lambda}^{W}e^{-\theta_{\lambda}t}R_{\lambda}^{W}(x) R_{\lambda}^{W}(-y) \qquad (t>0, x, y\in\alc)\\ \text{ with} \quad r^{W}_{\lambda} &= (\|R^{W}_{\lambda}\|^{2}_{L^{2}(\alc, \delta_{k})})^{-1}.
\end{align*}
The corresponding symmetric heat semigroup $(H^W\!(t))_{t\geq 0}$ on $(C(A_0), \|.\|_\infty)$ or on $(L^p(A_0),\delta_k)$  is generated by the corresponding closure of the  local operator $(L_{k}^{\operatorname{diff}},\mathcal T^W).$
Proceeding as in \cref{sec:poisson}, one can define a symmetric Poisson semigroup by subordination and obtain the symmetric Poisson kernel
\begin{align}\label{Poissonsymm}
  P_{k}^{W}\!(t, x, y) = \sum_{\lambda\in\domwt}r_{\lambda}^{W}e ^{-t\sqrt{\theta_{\lambda}}}R_{\lambda}^{W}(x)R_{\lambda}^{W}(-y) \qquad (t>0, x,y \in\alc).
\end{align}
The symmetric Poisson semigroup on $(C(A_0), \|.\|_\infty)$ is  given by
\begin{align*}
  P^{W}\!(t)f(x) := \int_{\alc}P_{k}^{W}(t, x, y) f(y)\,\delta_{k}(y)\, dy\, = \,\frac{1}{|W|}\int_{\T} P^{W}_{k}(t, x, y) f(y)\, \delta_{k}(y)\, dy 
\end{align*}
Let us now relate the non-symmetric and the symmetric setting.

\begin{lemma}\label{PW-P} \begin{enumerate}
\item[\rm{(a)}] $\Gamma_{k}(t, x,y)=\Gamma_{k}(t, wx, wy)$ and $P_{k}(t, x,y)=P_{k}(t, wx, wy)$ for all 
$w\in W.$  
\item[\rm{(b)}] Let $X$ be one of the spaces $(C(\mathbb T), \|.\|_\infty)$ or $L^p(\mathbb T, \delta_k), \, 1<p<\infty.$ Then for $f\in X^W$, we have $ H(t)f = H^W\!(t)f \text{ and } P(t)f = P^W\!(t)f.$
 \item[\rm{(c)}]  $\displaystyle  \Gamma_{k}^{W}(t, x, y)=\sum_{w\in W} \Gamma_{k}(t, x, wy)\quad\text{and}\quad P_{k}^{W}(t, x, y) = \sum_{w\in W}P_{k}(t, x, wy). $
 \end{enumerate}
\end{lemma}

\begin{proof} \rm{(a)} The non-symmetric heat semigroup satisfies 
  \begin{align*}
    H(t)(w.R_{\lambda}) = e^{-\theta_{\lambda}t} w.R_{\lambda} = w.H(t)R_{\lambda} 
  \end{align*}
  for all $\lambda\in P$. Consequently, $H(t)w = wH(t)$ on $\mathcal T$. By density, this extends to $X.$  Further, note that
  \begin{align*}
    \bigl(H(t)f\bigr)(w^{-1}x)=\int_{\T}\Gamma_{k}(t, w^{-1}x, y)f(y)\,\delta_{k}(y)\, dy
  \end{align*}
  and 
  \begin{align*}
    H(t)(w.f)(x) = \int_{\T}\Gamma_{k}(t, x, wy)f(y)\,\delta_{k}(y)\, dy. 
  \end{align*}
  This proves $\Gamma_{k}(t, wx, wy)=\Gamma_{k}(t, x, y),$ and subordination gives the corresponding property for the Poisson kernel $P_{k}$. 
    
    \rm{(b)} By general semigroup theory, the abstract Cauchy problem $\, L_k u = \partial_t u,\,  u(.\,,0) = f \in \mathcal T\,$ has a unique solution $u\in C^1([0, \infty[,X),$ which is given by $u(.\,,t) = H(t)f.$ Now suppose that $f\in \mathcal T$ is $W$-invariant. Then also $H(t)f$ is $W$-invariant for all $t>0.$ Therefore $v(.,t)= H(t)f $ provides the unique classical solution $v\in C^1([0,\infty[, X^W)$  of the abstract Cauchy problem $\, L_k^{\text{diff}} v = \partial_tv, \, v(.\,,0) = f.$ But on the other hand, this solution is given by $H^W\!(t) f$. The second identity follows by subordination.
    
     Part \rm{(c)} is immediate from (b).
    \end{proof}

For the symmetric Poisson semigroup we may similarly introduce $g$-functions
\begin{align*}
  g^{W}_{\bullet}(f)(x) \coloneqq \left(\int_{0}^{\infty} t|\bullet u_{f}^{W}(x, t)|^{2}\ dt\right)^{1/2}
\end{align*}
where $\bullet\in\{\partial_{t}, \partial_{\xi},\wt D_{\xi}, \nabla, \wt\nabla_{k}\}$ and $u_{f}^{W}(x, t)\coloneqq P^{W}(t)f(x)$ for $f\in L^{p}(\T, \delta_{k})^{W}$. We also define $g_0^W(f)$ as $g_0(f)$ before,  with $u_f^W$ instead of $u_f.$ 
However, as $u_{f}^{W}(\cdot, t)$ is $W$-invariant, we immediately see that 
\begin{align*}
  g^{W}_{0}(f)=0,\quad g^{W}_{\partial_{\xi}}(f)=g^{W}_{\wt D_{\xi}}(f),\quad g^{W}_{\nabla}(f)=g^{W}_{\wt\nabla_{k}}(f).
\end{align*}
Thus the study of the various $g$-functions from the previous sections reduces to the study of a single $g$-function of the classical form involving $\nabla$.
As a consequence of Lemma \ref{PW-P}(3)  we furthermore know that
\begin{align*}
  g_{\bullet}^{W}(f) = g_{\bullet}(f) \quad \text{ for }  f\in L^{p}(\T, \delta_{k})^{W}.
  \end{align*}
In particular, the $L^{p}$-boundedness for all $1< p< \infty$ for the $g$-functions corresponding to classical derivatives $\partial_{\xi}$ and $\nabla$ in the $W$-invariant case are an immediate consequence of  \cref{thm:gFunctionBound}\,(2). Further, by Lemma \ref{PW-P}, the results for the $g_{\ell}$-functions in \cref{thm:g1bound} and \cref{thm:gellbounds} carry over literally. We collect the results in the following

\begin{corollary}
  Let $1< p<\infty$ and $\bullet\in\{1,\ell, \nabla\}$ a subscript. Then there exists a constant $A_{\bullet, p}>0$ such that for all $f\in L ^{p}(\alc, \delta_{k})=L^{p}(\T, \delta_{k})^{W}$,
  \begin{align*}
    \|g_{\bullet}(f)\|_{p}\leq A_{\bullet, p}\|f\|_{p}. 
  \end{align*}
  Furthermore, there is a constant $B_{p}\geq 0$ such that for all $f\in L^{p}(\alc, \delta_{k})=L^{p}(\T, \delta_{k})^{W}$ with $\int_{\alc}f\,\delta_{k}\, dx=0,\,$ we have
  \begin{align*}
    \|f\|_{p}\leq B_{p}\|g_{1}(f)\|_{p}\leq B_{p}\|g_{\nabla}(f)\|_{p}. 
  \end{align*}
\end{corollary}

\begin{remark}
  Although we stated results for $W$-invariant functions, we extensively relied on non-symmetric theory. Indeed, Riesz transforms of $W$-invariant functions need not be $W$-invariant again. However, extending the $L^{p}$-boundedness of $g_{\wt\nabla_{k}}$ (and thus of $g_{\nabla}^{W}$) from $1<p\leq 2$, completely relied on the duality argument for the non-symmetric Riesz transforms. The subharmonicity in \cref{lem:gNablaKSubharmonic} further exemplifies this. If we would work in the symmetric setting and only consider $\nabla$ instead of $\wt\nabla_{k}$ (which are the same on $W$-invariant functions), we would run into structural problems, since the derivative of a $W$-invariant function is in general not $W$-invariant again. Thus, while additional reflection terms make certain calculations more laborious,  the non-symmetric theory is more flexible and adequate from a structural point of view.
\end{remark}

\begin{example}\label{example:rank1}
  \textbf{The symmetric rank-1 and direct product case.} Consider the root system $R=BC_1=\{\pm 1, \pm 2\}$ in $\frak a = \mathbb R$ with multiplicity $k=(k_1,k_2)$, where $k_1$ and $k_2$ are the values of $k$ on  $\pm 1$ and $\pm 2,$ respectively. In this case, $W=\Z_{2}, Q^\vee = P=\mathbb Z$ and $\mathbb T = \mathbb R/2\pi \mathbb Z.$ For $n \in \mathbb  Z_+=\{0,1,2,\ldots \}$ consider the (non-symmetric) Heckman-Opdam polynomials $R_n$ on $\mathbb T$ and put 
  $$ R_n^{\Z_{2}}(x)= \frac{1}{2}\bigl(R_n(x)+ R_{n}(-x)\bigr), \, x\in \mathbb T.$$
  The  $R_n^{\Z_{2}}$  are the normalized symmetric Heckman-Opdam polynomials of rank one. They can be written in terms of the classical Jacobi polynomials 
	$$ P_n^{(\alpha, \beta)}(z)	= \binom{n+\alpha}{n} \,_2F_1\Bigl(-n,n+\alpha+\beta+1; \alpha+1; \frac{1-z}{2}\Bigr)$$	
	as
  $$R_n^{\Z_{2}}(x)=R_n^{\Z_{2}}(k;x)=\frac{1}{\binom{n+\alpha}{n}} P_n^{(\alpha, \beta)}(\cos x) \quad \text{ with } \alpha= k_1+k_2-\tfrac{1}{2}, \, \beta = k_2-\tfrac{1}{2},$$ see \cite[Ex. 1.3.2]{HS94}.  Thus $ \rho = \frac{1}{2}(\alpha+\beta+1)$, and the  $R_n^{\Z_{2}}$ satisfy 
  $ L_k^{\operatorname{diff}}R_n^{\Z_{2}} = -n(n+2\rho) R_n^{\Z_{2}} \,$ 
	with 
	$$  L_k^{\text{diff}} \,=\, \frac{d^2}{dx^2} +\Bigl((\alpha-\beta) \cot\frac{x}{2} \, +\, (2\beta+1) \cot x\Bigr) \frac{d}{dx}.$$ 
Further, $$\delta_k(x) = \, 4^{\alpha+\beta+1}\big|\sin\bigl(\frac{x}{2}\bigr)\big|^{2\alpha+1}\big|\cos\bigl(\frac{x}{2}\bigr)\big|^{2\beta+1}. $$
The symmetric Poisson kernel is given by 
\begin{equation}\label{Poisson_Jacobi} P_k^{\Z_{2}}(t,x,y)= \sum_{n=0}^\infty r_n^{\Z_{2}}e^{-t\sqrt{n(n+\alpha+\beta+1)}} R_n^{\Z_{2}}(x)R_n^{\Z_{2}}(y). \end{equation}
	In \cite{NS08}, Littlewood-Paley theory for a multivariate direct product variant of this Jacobi setting was developed. 
Within  our general framework, this corresponds to the $W$-invariant case associated with the root system $R=\{\pm e_i, \pm 2e_i: 1\leq i \leq n\}$ in $\mathfrak a=\mathbb R^n$,  where $W=\mathbb Z_2^n$ acts by sign changes of the coordinates. The weight lattice is $P=\mathbb Z^n,$ and a multiplicity function on $R$ is of the form $k=(k^1, \ldots, k^n)$ with $k^i = (k^i_1, k^i_2),$ where $k_1^i$ and $k_2^i$ denote the values of $k$ on $\pm e_i$ and  $\pm 2e_i$, respectively. The corresponding symmetric Heckman-Opdam polynomials  are indexed by $P_+= \mathbb Z_+^n$ and are obtained from the non-symmetric Heckman-Opdam polynomials $(R_\lambda)_{\lambda \in P}$ by symmetrization: 
$$ R_\lambda^{\Z_{2}^{n}}(x) = \,\frac{1}{2^n}\!\sum_{w\in \mathbb Z_2^n} R_\lambda(wx), \quad \lambda \in \mathbb Z_+^n.$$ 
	Due to the direct product structure of the Weyl group, these are just tensor products of one-variable Jacobi polynomials: For $\lambda=(\lambda_1, \ldots, \lambda_n)\in \mathbb Z_+^n$ and $x=(x_1, \ldots, x_n) \in \mathbb R^n/2\pi\mathbb Z^n,$
  $$ R_\lambda^{\Z_{2}^{n}}(x) =\,\prod_{i=1}^n R_{\lambda_i}^{\Z_{2}}(k^i;x_i)$$
  with the one-variable Jacobi polynomials $R_{n}^{\Z_{2}}(k;x)$ as introduced above. 	The $\mathbb Z_2^n$-symmetric Poisson kernel \eqref{Poissonsymm} coincides with the kernel of the Poisson semigroup studied in \cite{NS08}.	
	The authors in \cite{NS08} work with a multivariate Jacobi operator $\mathcal J^{(\alpha, \beta)}$ which is obtained from  $-L_k^{\text{diff}}$ by the coordinate transform $x_i\mapsto \cos x_i,$ but instead of Cherednik operators they involve the first-order differential operators $\delta_i= \sqrt{1-x_i^2}\,\partial_{x_i}$ and their formal adjoints $\delta_i^\ast$, satisfying $\mathcal J^{(\alpha, \beta)} = \sum_{i=1}^n \delta_i^\ast\delta_i.$ They study  $g$-functions for the symmetric Poisson semigroup as well as Riesz transforms, following the lines of \cite{St70}. 
Our approach involving Cherednik operators includes the main results of \cite{NS08}. It has the advantage that it works for arbitrary root systems in a uniform manner.	
	
	We finally point out that in  the literature on harmonic analysis associated with one-variable Jacobi expansions (\cite{MS65, NS12, NS13}), the following variant of the  kernel \eqref{Poisson_Jacobi} is more common:
  $$ \widetilde P(t,x,y) = \sum_{n=0}^\infty r_n^{\Z_{2}} e^{-t(n+\rho)}R_n^{\Z_{2}}(x)R_n^{\Z_{2}}(y). $$
	 For fixed $y$, the kernel $u(x,t) = \widetilde P(t,x,y)$ satisfies the 
	 Laplace-type equation   $(L_k^{\text{diff}} - \rho^2 + \partial_t^2) u = 0.$ 
The advantage of our setting and that of \cite{NS08} is that our Poisson semigroup is Markovian, not only Submarkovian as for $\widetilde P.$
	\end{example}

\begin{example}\label{example:group_case}
\textbf{The case of Riemannian symmetric spaces of the compact type.}
After passing to averages over the Weyl group $W$, our framework also covers Littlewood-Paley-Stein theory for radial functions on Riemannian symmetric spaces of the compact type. Let us briefly elaborate: Let $U/K$ be a Riemannian symmetric space of the compact type, where $U$ is a simply connected compact Lie group and $K$ is the fixed group of some involutive automorphism of $U.$ The derivation of $\theta$ gives an involution of the Lie algebra $\mathfrak u$ of $U$ with eigenspace decomposition $\mathfrak u = \mathfrak k + i \mathfrak p,$ corresponding to the eigenvalues $1$ and $-1$. Functions on $U/K$ which are radial, i.e. $K$-invariant, can be naturally considered as functions on $i \frak a$, where $\frak a$ is a maximal abelian subspace of $\mathfrak p.$ The space $\mathfrak a$ is a finite-dimensional Euclidean space with the Killing form as scalar product. The (zonal) spherical functions of $U/K$, which form the basic building blocks for radial harmonic analysis on $U/K$ are then given in terms of the symmetric Heckman-Opdam polynomials associated with the root data of $(U,K).$ For details, see e.g. \cite{RR15}. Radial analysis on $U/K$ then 
	boils down to the analysis of Heckman-Opdam expansions.
\end{example}

\section*{AI Disclosure} 
During the preparation of this manuscript, the authors used OpenAI's GPT-5.6 Sol and Claude Opus 5 solely to assist with proofreading and minor language improvements. All mathematical content, in particular the revised proofs, was developed and written by the authors independently without use of AI-assisted generation tools.


\begin{thebibliography}{9999999}
\bibitem[AS12]{AS12} 
B. Amri and M. Sifi, Riesz transforms for Dunkl transform, Ann. Math. Blaise Pascal {\bf 19} (2012), no.~1, 247--262.

\bibitem[A17]{A17} 
J.-P. Anker, An introduction to Dunkl theory and its analytic aspects, in {\it Analytic, algebraic and geometric aspects of differential equations}, 3--58, Trends Math., Birkh\"auser/Springer, Cham, 2017.

\bibitem[ADH19]{ADH19} 
J.-P. Anker, J. Dziuba\'nski and A. Hejna, Harmonic functions, conjugate harmonic functions and the Hardy space $H^1$ in the rational Dunkl setting, J. Fourier Anal. Appl. {\bf 25} (2019), no.~5, 2356--2418.

\bibitem[BS11]{BS11}
N. Ben~Salem and T. Samaali, Hilbert transform and related topics associated with Jacobi-Dunkl operators of compact and noncompact types, Adv. Pure Appl. Math. {\bf 2} (2011), no.~3-4, 367--388.

\bibitem[B24]{B24diss} D. Brennecken,
Contributions to Dunkl theory,
Universit\"at Paderborn, Dissertation, 2024.

\bibitem[B25]{B25} D. Brennecken, Boundedness of the Cherednik kernel and its limit transition from type BC to type A, Indag. Math. (N.S.) {\bf 36} (2025), no.~6, 1717--1744.

\bibitem[CS79]{CS79} W. Connett and A. Schwartz, The Littlewood-Paley theory for Jacobi expansions. \emph{Trans. Amer. Math. Soc.} {\bf 251} (1979), 219--234. 

\bibitem[Ch91]{Ch91} 
I. Cherednik, A unification of Knizhnik-Zamolodchikov and Dunkl operators via affine Hecke algebras, Invent. Math. {\bf 106} (1991), no.~2, 411--431.

\bibitem[DW26]{DW26} F. D'Emilio and B. Wick, Martingale transforms and compensated Bellman estimates for Dunkl Riesz transforms. ArXiv:2609.12117v1. 


\bibitem[dJ93]{dJ93} 
M. de~Jeu, The Dunkl transform, Invent. Math. {\bf 113} (1993), no.~1, 147--162.

\bibitem[DH22]{DH22} 
J. Dziuba\'nski and A. Hejna, Upper and lower bounds for Littlewood-Paley square functions in the Dunkl setting, Studia Math. {\bf 262} (2022), no.~3, 275--303.
  
\bibitem[DS58]{DS58} N. Dunford and J.~T. Schwartz, Linear Operators Part I: General Theory. Pure and Applied Mathematics, Interscience Publishers, inc., New York, 1958.

  \bibitem[Du89]{Du89}
C.~F. Dunkl, Differential-difference operators associated to reflection groups, Trans. Amer. Math. Soc. {\bf 311} (1989), no.~1, 167--183.

\bibitem[EK86]{EK86} S. Ethier and T. Kurtz, Markov Processes: Characterization and Convergence. Wiley, New York, 1986. 

\bibitem[FSS15]{FSS15}
L. Forzani, E. Sasso and R. Scotto, $L^p$ boundedness of Riesz transforms for orthogonal polynomials in a general context, Studia Math. {\bf 231} (2015), no.~1, 45--71.

\bibitem[H00]{H00} S. Helgason, \emph{Groups and Geometric Analysis: Integral Geometry, Invariant Differential Operators, and Spherical Functions.} Mathematical Surveys and Monographs. American Mathematical Society, 2000.  

\bibitem[HO21]{HO21} G. Heckman and E. Opdam,  Jacobi polynomials and hypergeometric functions associated with root systems. In: Encyclopedia of Special Functions, Part II: Multivariable Special Functions, eds. T.H. Koornwinder, J.V. Stokman, Cambridge University Press, Cambridge, 2021. 
 
\bibitem[HS94]{HS94} G. Heckman and  H. Schlichtkrull, Harmonic Analysis and Special Functions on Symmetric Spaces, Perspectives in Mathematics, Vol. 16, Academic Press, 1994. 
 
 \bibitem[H23]{H23}
 A. Hejna, Dimension-free $L^p$-estimates for vectors of Riesz transforms in the rational Dunkl setting, J. Anal. Math. {\bf 150} (2023), no.~2, 485--528.

 \bibitem[K25]{K25} 
V. Kumar, $L^p$-$L^q$ hypergeometric spectral and Fourier multipliers associated with root systems, Potential Anal. {\bf 63} (2025), no.~3, 1517--1538.

\bibitem[LZL17]{LZL17}
J.~Q. Liao, X. Zhang and Z. Li, On Littlewood-Paley functions associated with the Dunkl operator, Bull. Aust. Math. Soc. {\bf 96} (2017), no.~1, 126--138.

\bibitem[LZ23]{LZ23} 
H. Li and M. Zhao, Dimension-free square function estimates for Dunkl operators, Math. Nachr. {\bf 296} (2023), no.~3, 1225--1243.

\bibitem[MS65]{MS65} 
B. Muckenhoupt and E.~M. Stein, Classical expansions and their relation to conjugate harmonic functions, Trans. Amer. Math. Soc. {\bf 118} (1965), 17--92.

\bibitem[NS08]{NS08} 
A.~H. Nowak and P. Sj\"ogren, Riesz transforms for Jacobi expansions, J. Anal. Math. {\bf 104} (2008), 341--369.

\bibitem[NS12]{NS12} 
A.~H. Nowak and P. Sj\"ogren, Calder\'on-Zygmund operators related to Jacobi expansions, J. Fourier Anal. Appl. {\bf 18} (2012), no.~4, 717--749.

\bibitem[NS13]{NS13} 
A.~H. Nowak and P. Sj\"ogren, Sharp estimates of the Jacobi heat kernel, Studia Math. {\bf 218} (2013), no.~3, 219--244.

\bibitem[NSt06]{NSt06} 
A.~H. Nowak and K. Stempak, $L^2$-theory of Riesz transforms for orthogonal expansions, J. Fourier Anal. Appl. {\bf 12} (2006), no.~6, 675--711.

\bibitem[Op95]{Op95} 
E.~M. Opdam, Harmonic analysis for certain representations of graded Hecke algebras, Acta Math. {\bf 175} (1995), no.~1, 75--121.

\bibitem[Op00]{Op00} 
E.~M. Opdam, {\it Lecture notes on Dunkl operators for real and complex reflection groups}, MSJ Memoirs, 8, Math. Soc. Japan, Tokyo, 2000.

\bibitem[RR11]{RR11}
H. Remling and M. R\"osler, The heat semigroup in the compact Heckman-Opdam setting and the Segal-Bargmann transform, Int. Math. Res. Not. IMRN {\bf 2011}, no.~18, 4200--4225.

\bibitem[RR15]{RR15}
H. Remling and M. R\"osler, Convolution algebras for Heckman-Opdam polynomials derived from compact Grassmannians, J. Approx. Theory {\bf 197} (2015), 30--48.

\bibitem[RKV13]{RKV13}
M. R\"osler, T.~H. Koornwinder and M. Voit, Limit transition between hypergeometric functions of type BC and type A, Compos. Math. {\bf 149} (2013), no.~8, 1381--1400.

\bibitem[RV04]{RV04} 
M. R\"osler and M. Voit, Positivity of Dunkl's intertwining operator via the trigonometric setting, Int. Math. Res. Not. {\bf 2004}, no.~63, 3379--3389.

\bibitem[RV08]{RV08} M. R\"osler and M. Voit, Dunkl theory, convolution algebras, and related Markov processes. In: Harmonic and stochastic analysis of Dunkl processes; Eds. P. Graczyk et al.,  \emph{Hermann Mathematiques}, Paris, 2008, pp. 1--112.

\bibitem[Sch08]{Sch08} 
B. Schapira, Contributions to the hypergeometric function theory of Heckman and Opdam: sharp estimates, Schwartz space, heat kernel, Geom. Funct. Anal. {\bf 18} (2008), no.~1, 222--250.

\bibitem[S00a]{S00a}
S. Sahi, A new formula for weight multiplicities and characters, Duke Math. J. {\bf 101} (2000), no.~1, 77--84.

\bibitem[S00b]{S00b}
S. Sahi, Some properties of Koornwinder polynomials, in {\it $q$-series from a contemporary perspective (South Hadley, MA, 1998)}, 395--411, Contemp. Math., 254, Amer. Math. Soc., Providence, RI. 

\bibitem[S05]{S05}
F. Soltani, Littlewood-Paley operators associated with the Dunkl operator on $\Bbb R$, J. Funct. Anal. {\bf 221} (2005), no.~1, 205--225.

\bibitem[St70]{St70} 
E.~M. Stein, {\it Topics in harmonic analysis related to the Littlewood-Paley theory}, Annals of Mathematics Studies, No. 63, Princeton Univ. Press, Princeton, NJ, 1970 Univ. Tokyo Press, Tokyo, 1970.

\bibitem[StT03]{ST03} 
K. Stempak and J.~L. Torrea, Poisson integrals and Riesz transforms for Hermite function expansions with weights, J. Funct. Anal. {\bf 202} (2003), no.~2, 443--472.

\bibitem[Th93]{Th93} 
S. Thangavelu, {\it Lectures on Hermite and Laguerre expansions}, Mathematical Notes, 42, Princeton Univ. Press, Princeton, NJ, 1993.

\bibitem[TX07]{TX07} S. Thangavelu and Y. Xu, Riesz transform and Riesz potentials for Dunkl transform, J. Comput. Appl. Math. {\bf 199} (2007), no.~1, 181--195.

\end{thebibliography}
\end{document}